\documentclass[12pt]{amsart}
\usepackage{latexsym}
\usepackage{amssymb} 
\usepackage{mathrsfs}
\usepackage{amsmath}
\usepackage{latexsym}
\usepackage{delarray}
\usepackage{amssymb,amsmath,amsfonts,amsthm,
mathrsfs}

\usepackage{hyperref}
\renewcommand\eqref[1]{(\ref{#1})} 

\newtheorem{theorem}{Theorem}[section]
\newtheorem{corollary}[theorem]{Corollary}
\newtheorem{lemma}[theorem]{Lemma}
\newtheorem{proposition}[theorem]{Proposition}
\newtheorem{definition}[theorem]{Definition}
\theoremstyle{definition}
\newtheorem{remark}[theorem]{Remark}
\newtheorem{example}[theorem]{Example}

\newcommand{\wt}[1]{\widetilde{#1}}

\newcommand{\Cinfc}{\ensuremath{\mathcal{C}^\infty_{\text{c}}}}
\newcommand{\D}{\ensuremath{{\mathcal D}}}
\renewcommand{\S}{\mathscr{S}}
\newcommand{\E}{\ensuremath{{\mathcal E}}}

\newcommand{\mb}[1]{\ensuremath{\mathbb{#1}}}
\newcommand{\N}{\mb{N}}

\newcommand{\R}{\mb{R}}

\newcommand{\lara}[1]{\langle #1 \rangle}

\newfont{\bigmath}{cmr12 at 13pt}

\newfont{\grecomath}{cmmi12 at 15pt}

\newfont{\bl}{msbm10 scaled \magstep2}

\newcommand{\beq}{\begin{equation}}
\newcommand{\eeq}{\end{equation}}

\newcommand{\notmid}{\mid\kern-0.5em\not\kern0.5em}

\newcommand{\inner}[3][\empty]{\ifx#1\empty\left( #2|#3\right)\else#1( #2|#3 #1)\fi}

\newcommand{\eps}{\varepsilon}

\renewcommand\N{{\mathbb N}_0}

\title[On wave equation]
{On the wave equation with multiplicities and space-dependent irregular coefficients}

\author[Claudia Garetto]{Claudia Garetto}
\address{
  Claudia Garetto:
  \endgraf
  Department of Mathematical Sciences
  \endgraf
  Loughborough University
  \endgraf
  Loughborough, Leicestershire, LE11 3TU
  \endgraf
  United Kingdom
  \endgraf
  {\it E-mail address} {\rm c.garetto@lboro.ac.uk}
  }

\date{}


\begin{document}

\maketitle

\begin{abstract}
In this paper we study the well-posedness of the Cauchy problem for a wave equation with multiplicities and space-dependent irregular coefficients. 
As in \cite{GR:14}, in order to give a meaningful notion of solution, we employ the notion of very weak solution, which construction is based on a parameter dependent regularisation
of the coefficients via mollifiers. We prove that, even with distributional coefficients, a very weak solution exists for our Cauchy problem and it converges to the classical one when the coefficients are smooth. 
The dependence on the mollifiers of very weak solutions is investigated at the end of the paper in some instructive examples.
\end{abstract}

\section{Introduction}
We want to study the well-posedness of the Cauchy problem
\beq
\label{CP_wave}
\begin{split}
\partial_t^2u-\sum_{i=1}^na_i(x)\partial^2_{x_i} u&=f(t,x),\quad t\in[0,T], x\in\R^n,\\
u(0,x)&=g_0,\\
\partial_tu(0,x)&=g_1,
\end{split}
\eeq
where all the functions involved are real valued and $a_i(x)\ge 0$ for all $x\in\R^n$ and $i=1,\dots,n$.  We assume that the coefficients $a_i$  and the right-hand side $f$ are irregular, i.e., discontinuous or more in general distributions. Typical examples are given by the coefficients $a_i$ being jump functions (Heaviside) or positive distributions. This kind of equation appears naturally in geophysics when modelling the propagation of waves (acoustic waves, seismic waves, etc..) through a multi-layered medium (e.g. subsoil), see \cite{AdHSU:08, BdHSU:11, CdHKU:19, dHHO:08, dHHSU:12, HdH01, HdH02} and references therein. In \cite{GR:14} the author and M. Ruzhansky have studied second order hyperbolic equations with non-regular coefficients only depending on time. Here we assume that the discontinuity is spatial. This requires totally different techniques with respect to the ones employed in \cite{GR:14} and widely enlarges the family of physical problems we can treat. In addition, since the coefficients $a_i$ might vanish we allow our problem to have multiplicities. 

It is a notoriously difficult problem to prove the well-posedness of the Cauchy problem for a hyperbolic equation or system with multiplicities, since the presence of multiplicities might have a disruptive effect and leads to results of non-existence or non-uniqueness (see for instance \cite{CS} and \cite{CJS2}). Many authors have proven Gevrey well-posedness for hyperbolic Cauchy problems with multiplicities mainly when the coefficients depend only on time (see \cite{ColKi:02, ColKi:02-2, CDS, dAKi:05, GR:11, GR:12, GarRuz:3, KS} and references therein). The methods employed to obtain Gevrey well-posedness often combine algebraic transformations like symmetrisation or quasi-symmetrisation and energy estimates and require specific conditions of Levi-type on the lower order terms \cite{dASpa:98, GR:12, GarJ}. $C^\infty$ well-posedness has also been obtained for some special classes of $t$-dependent equations and systems, for instance in \cite{GarRuz:7, JT}. A full understanding of hyperbolic Cauchy problems with multiplicities and time- and space-dependent coefficients is still missing. Some special cases have been analysed but often this means to impose strong assumptions on the multiplicities themselves and the regularity of the coefficients. In the case of space dependent coefficients we recall the celebrated work of Bronshtein \cite{B} for hyperbolic equations with multiplicities. This result was extended to $(t,x)$ scalar equations by Ohya and Tarama  \cite{OT:84} and to systems by Kajitani and Yuzawa \cite{KY:06}. In all these results the regularity assumptions on the coefficients are always quite strong with respect to $x$  (Gevrey) and do not go below H\"older in $t$. It is unclear which kind of well-posedness once could get when the regularity in the space variable is less than Gevrey and which hypotheses on coefficients, eigenvalues, multiplicities and lower order terms will lead to $C^\infty$ well-posedness in the general $(t,x)$-dependent case. Finally, we recall that some geometrical and microlocal analytic approach to hyperbolic equations and systems with multiplicities has been employed by authors as Melrose and Uhlmann \cite{MU:1,MU:2}, Dencker \cite{Dencker}, H\"ormander \cite{Hord},
Ivrii and Petkov \cite{IvPet}, Kamotski and Ruzhansky \cite{KR2}, Kucherenko and Osipov \cite{9}, Mascarello and Rodino \cite{MR}, Parenti and Parmeggiani \cite{PP}, Yagdjian \cite{Yagd}, to quote a few and more recently by the author in collaboration with J\"ah and Ruzhanskly to prove well-posedness in anisotropic Sobolev spaces for non-diagonalisable hyperbolic systems with multiplicities in \cite{GJR1, GJR2}. 

In this paper we want to study the Cauchy problem \eqref{CP_wave} for the $x$-dependent wave equation in presence of \emph{both singularities and multiplicities}. We recall that a Cauchy problem is $C^\infty$ (globally) well-posed if for any smooth right-hand side $f$ and initial data $g_0, g_1$ it has a unique solution $u\in C^\infty([0,T]\times\R^n)$. In the sequel we give a short survey on the \emph{status of the art} for the Cauchy problem above, by discussing the results in the literature and the standing open problems.

Second order weakly hyperbolic equations (i.e. with multiplicities) have been stu\-di\-ed by Oleinik in her pioneristic work \cite{O70}. Here she proved that the Cauchy problem for a second order hyperbolic operator in variational form,
\begin{multline*}
Lu=u_{tt}-\sum_{i,j=1}^n (a_{ij}(t,x)u_{x_j})_{x_i}+\sum_{i=1}^n[(b_i(t,x)u_{x_i})_t+b_i(t,x)u_t)_{x_i}]\\
+c(t,x)u_t+\sum_{i=1}^n d_i(t,x)u_{x_i}+e(t,x)u
\end{multline*}
with smooth and bounded coefficients, i.e., coefficients in $B^\infty([0,T]\times\R^n)$, the space of smooth functions with bounded derivatives of any order $k\ge 0$, is $C^\infty$ well-posed provided that the lower order terms fulfil a specific Levi condition known nowadays as Oleinik's condition: there exist $A,C>0$ such that 
\[
\biggl[\sum_{i=1}^n d_i(t,x)\xi_i\biggr]^2\le C\biggl\{A\sum_{i,j=1}^na_{ij}(t,x)\xi_i\xi_j-\sum_{i,j=1}^n\partial_ta_{ij}(t,x)\xi_i\xi_j\biggr\},
\]
for all $t\in[0,T]$ and $x,\xi\in\R^n$. Note that Oleinik's condition is automatically fulfilled when the coefficients of the principal part are independent of $t$ and the $d_i$'s vanish identically.

Since 
\[
\partial_t^2u-\sum_{i=1}^na_i(x)\partial^2_{x_i} u
\]
can be written as 
\[
Lu=\partial_t^2u-\partial_{x_i}\biggl(\sum_{i=1}^na_i(x)\partial_{x_i} u\biggr)+\sum_{i=1}^n\partial_{x_i}a_i(x)\partial_{x_i}u
\]
we have that in our specific case the Oleinik's condition is formulated as follows:
\beq
\label{est_OL}
\biggl[\sum_{i=1}^n \partial_{x_i}a_i(x)\xi_i\biggr]^2\le CA\sum_{i=1}^n a_i(x)\xi_i^2.
\eeq
The estimate \eqref{est_OL} holds automatically by Glaeser's inequality if the coefficients $a_i$ are positive, at least of class $C^2$ and with bounded second order derivatives. Indeed (Glaeser's inequality), 
\emph{if $a\in C^2(\R^n)$, $a(x)\ge 0$ for all $x\in\R^n$ and 
\[
\sum_{i=1}^n\Vert \partial^2_{x_i}a\Vert_{L^\infty}\le M,
\]
for some constant $M>0$. Then,  
\[
|\partial_{x_i}a(x)|^2\le 2M a(x),
\]
for all $i=1,\dots,n$ and $x\in\R^n$.}
 
We therefore conclude that 
\begin{itemize}
\item the Cauchy problem \eqref{CP_wave} is $C^\infty$ well-posed if the coefficients $a_i$ are positive and \emph{regular}, i.e. smooth with bounded derivatives of any order.
\end{itemize}
It is natural to ask what happens to the $C^\infty$ well-posedness when we drop the assumption of regularity on the coefficients $a_i$. The introduction of non-regular coefficients, for instance discontinuous coefficients, entails some foundational problems like how to define a meaningful notion of solution. In view of the famous Schwartz impossibility result on multiplication of distributions we might fail to give a meaning to the pro\-duct $a_i(x)\partial_{x_i}^2u$ when $a_i$ is a discontinuous function or more in general a distribution. It is therefore reasonable to follow the approach adopted for hyperbolic equations with irregular $t$-dependent coefficients in \cite{GR:14}. Mainly, to regularise the original problem via convolution with suitable mollifiers and to look for nets of smooth solutions. This means to find a \emph{very weak} solution for the Cauchy problem \eqref{CP_wave} and therefore to prove well-posedness in the \emph{very weak} sense. 
Note that 
\begin{itemize}
\item no results of very weak well-posedness are known for hyperbolic equations with multiplicities and irregular coefficients depending on $x$. 
\end{itemize}
Symmetric hyperbolic systems have been investigated by Lafon and Oberguggenberger in \cite{LO:91} but the wave equation we are studying here cannot be easily transformed into a symmetric system due to its multiplicities. Here we combine the regularisation methods employed in \cite{GR:14} with the symmetrisation techniques for hyperbolic systems with multiplicities employed in \cite{J:09, JT, ST}. In particular, we take inspiration from the study of weakly hyperbolic homogeneous equations done by Spagnolo and Taglialatela in \cite{ST} in the 1-space dimensional case. Very weak solutions have been recently employed to study different classes of PDEs with singularities in \cite{RT, RY}. A numerical analysis of very weak solutions has also been initiated in \cite{ART:19, MRT:19} for the wave equation with irregular time-dependent coefficients leading to the discovery of new interesting echo-effects (see \cite{MRT:19, SW}). This paper is the needed theoretical background for the numerical analysis of very weak solutions when discontinuities appear in space rather than in time. 

The {\bf main result} of the paper is the proof that

\emph{the Cauchy problem \eqref{CP_wave}
\[
\begin{split}
\partial_t^2u-\sum_{i=1}^na_i(x)\partial^2_{x_i} u&=f(t,x),\quad t\in[0,T], x\in\R^n,\\
u(0,x)&=g_0,\\
\partial_tu(0,x)&=g_1,
\end{split}
\]
where  $a_i\in\E'(\R^n)$ with $a_i\ge 0$ for all $i=1,\dots ,n$, $f\in C^\infty([0,T],\E'(\R^n))$ and $g_0,g_1\in\E'(\R^n)$, is \emph{very weakly well-posed}, i.e., a very weak solution exists and it is unique in a suitable sense (see Definition \ref{vw_well_posed}).} Note that our result recovers the classical $C^\infty$ well-posedness when the coefficients $a_i$ are regular, i.e., smooth.  

The paper is organised as follows.

In Section \ref{sec_prelim} we recall some basic notions concerning very weak solutions and the technical methods (Glaeser's inequality) employed in proving the well-posedness of the Cauchy problem \eqref{CP_wave}. As a preparatory work and to better explain the strategy adopted in our proof we study the Cauchy problem \eqref{CP_wave} in one-space dimension in Section \ref{sec_1}. In Section \ref{sec_CP_smooth} we prove the $C^\infty$ well-posedness of the Cauchy problem \eqref{CP_wave} in case of smooth coefficients. We provide here an alternative proof to the one given by Oleinik in \cite{O70} which has the potential to be extended to higher order hyperbolic equations,  as done in \cite{ST} in the 1-dimensional case. In Section \ref{sec_CP_vw} we prove the very weak well-posedness of the Cauchy problem \eqref{CP_wave} in presence of non regular coefficients. Finally, in Section \ref{sec_CP_appl} we show that our result recovers the classical one when the coefficients are smooth and we discuss some interesting physical examples by analysing the dependence of the solution on the mollifiers employed in the regularisation process.

\section{Preliminaries}
\label{sec_prelim}

\subsection{Very weak solutions and very weak well-posedness}
In the sequel, $\varphi$ is a mollifier, i.e., a compactly supported smooth function with $\int_{\R^n}\varphi(x)\, dx=1$. Given a positive function $\omega(\eps)$  converging to $0$ as $\eps\to 0$ we define the regularising net 
\[
\varphi_{\omega(\eps)}(x)=\omega(\eps)^{-1}\varphi(x/\omega(\eps)).
\]
A distribution $u\in\E'(\R^n)$ can be regularised by convolution with the mollifier $\varphi_{\omega(\eps)}$. This generates a net of smooth functions whose properties will be investigated in the sequel. For a detailed proof of the following results we refer the reader to \cite{GR:14, GKOS:01, Oberguggenberger:Bk-1992}  and reference therein.
\begin{proposition}
\label{prop_reg_nets}
\leavevmode
\begin{itemize}
\item[(i)] If $u\in\mathcal{E}'(\R^n)$ then there exists $N\in\N$ and for all $\alpha\in\N^n$ there exists $c>0$ such that 
\[
|\partial^\alpha(u\ast\varphi_{\omega(\eps)})(x)|\le c\,\omega(\eps)^{-N-|\alpha|},
\]
for all $x\in\R^n$ and $\eps\in(0,1]$.
\item[(ii)] If $f\in C^\infty_c(\R^n)$ then for all $\alpha\in\N^n$ there exists $c>0$ such that 
\[
|\partial^\alpha(f\ast\varphi_{\omega(\eps)})(x)|\le c 
\]
for all $x\in\R^n$ and $\eps\in(0,1]$.
\item[(iii)] If $f\in C^\infty_c(\R^n)$ and the mollifier $\varphi$ has all the moments vanishing, i.e., $\int\varphi(x)\, dx=1$ and $\int_{\R^n}x^\alpha\varphi(x)\, dx=0$ for all multi-index $\alpha$ with $|\alpha|>1$,  then for all $\alpha\in\N^n$ and for all $q\in\N$ there exists $c>0$ such that 
\[
|\partial^\alpha(f\ast\varphi_{\omega(\eps)}(x)-f(x))|\le c\,\omega(\eps)^q,
\]
for all $x\in\R^n$ and $\eps\in(0,1]$.
\end{itemize}
\end{proposition}
\begin{remark}
\label{rem_net}
Since $\omega(\eps)$ tends to $0$ as $\eps\to 0$ it is not restrictive to assume that it is bounded. If there exists $c_1,c_2>0$ and $r>0$ such that 
\[
c_2\eps^r\le \omega(\eps)\le c_1,
\]
for all $\eps\in(0,1]$ then $\omega(\eps)$ can be replaced with $\eps$ in the estimates in Proposition \ref{prop_reg_nets}. For the sake of simplicity we will say that $\omega(\eps)$ is a \emph{positive scale}.
\end{remark}
\begin{remark}
\label{rem_non_compact}
The estimates of Proposition \ref{prop_reg_nets} can be locally extended to $u\in\D'(\R^n)$ and $f\in C^\infty(\R^n)$. This can be done by introducing an open covering of $\R^n$, corresponding cut-off functions and then a partition of unity subordinated to the covering. For the technical details we refer to Section 1.2.2 in \cite{GKOS:01}.  
\end{remark}
In the sequel, the notation $K\Subset\R^n$ means that $K$ is a compact set in $\R^n$. We now consider nets of smooth functions and we introduce the notions of $C^\infty$-moderate net and 
$C^\infty$-negligible net. 
\begin{definition}
\label{def_mod_neg}
\leavevmode
\begin{itemize}
\item[(i)] A net $(v_\eps)_\eps\in C^\infty(\R^n)^{(0,1]}$ is $C^\infty$-moderate if for all $K\Subset\R^n$ and for all $\alpha\in\N^n$ there exist $N\in\N$ and $c>0$ such that
\beq
\label{mod_C_inf}
|\partial^\alpha v_\eps(x)|\le c\eps^{-N},
\eeq
uniformly in $x\in K$ and $\eps\in(0,1]$.
\item[(ii)] A net $(v_\eps)_\eps\in C^\infty(\R^n)^{(0,1]}$  is $C^\infty$-negligible if for all $K\Subset\R^n$, $\alpha\in\N^n$ and $q\in\N$ there exists $c>0$ such that
\beq
\label{neg_C_inf}
|\partial^\alpha v_\eps(x)|\le c\eps^{q},
\eeq
uniformly in $x\in K$ and $\eps\in(0,1]$.
\end{itemize}
\end{definition}
Since in this paper we will only consider nets of smooth functions we can drop the suffix $C^\infty$- and adopt the simpler expressions \emph{moderate net} and \emph{negligible net}. The sets of moderate nets and the set of negligible nets are both
differential algebras, i.e. algebras closed with respect to derivation. From Proposition \ref{prop_reg_nets} and Remark \ref{rem_net} we easily obtain the following proposition which investigates how the regularisation depends on the choice of the mollifier and 
the $\eps$-scale. In other words, we prove that a change of mollifier and scale does not have an effect on the asymptotic behaviour of the regularisation as $\eps$ tends to $0$.
\begin{proposition}
\label{prop_nets_asymp}
Let $\omega(\eps), \omega_1(\eps),\omega_2(\eps)$ be positive scales. Let $\varphi, \varphi_1, \varphi_2$ be mollifiers. 
\begin{itemize}
\item[(i)] If $u\in\E'(\R^n)$ then $(u\ast\varphi_{\omega(\eps)})_\eps$ tends to $u$ as $\eps\to 0$ in the sense of distributions, i.e.,  
\[
\lim_{\eps\to 0}\int_{\R^n}u\ast\varphi_{\omega(\eps)}(x)\psi(x)\, dx=u(\psi),
\]
for all $\psi\in C^\infty(\R^n)$.
\item[(ii)] If $u\in\E'(\R^n)$ then $(u\ast\varphi_{\omega_1(\eps)}-u\ast\varphi_{\omega_2(\eps)})_\eps$ tends to $0$ in the sense of distributions.
\item[(iii)] If $u\in\E'(\R^n)$ then $(u\ast\varphi_{1,\omega(\eps)}-u\ast\varphi_{2,\omega(\eps)})_\eps$ tends to $0$ in the sense of distributions.
\end{itemize}
Assume in addition that the mollifiers have all the moments vanishing.
\begin{itemize}
\item[(iv)] If $f\in C^\infty_c(\R^n)$ then the net $(f\ast\varphi_{\omega(\eps)}-f)_\eps$ is negligible.
\item[(v)] If $f\in C^\infty_c(\R^n)$ then the net $(f\ast\varphi_{\omega_1(\eps)}-f\ast\varphi_{\omega_2(\eps)})_\eps$ is negligible.
\item[(vi)] If $f\in C^\infty_c(\R^n)$ then the net $(f\ast\varphi_{1,\omega(\eps)}-f\ast\varphi_{2,\omega(\eps)})_\eps$ is negligible.
\end{itemize}

\end{proposition}
\begin{proof}
\leavevmode
\begin{itemize}
\item[(i)] By definition of mollifier we have that $\varphi_{\omega(\eps)}$ tends to $\delta$ in the sense of distributions as $\eps\to 0$. Hence, by continuity of the convolution product we conclude that $(u\ast\varphi_{\omega(\eps)})_\eps$ tends to $u$ 
in $\E'(\R^n)$.
\item[(ii)] It follows from (i) by writing 
\[
u\ast\varphi_{\omega_1(\eps)}-u\ast\varphi_{\omega_2(\eps)}=(u\ast\varphi_{\omega_1(\eps)}-u)-(u\ast\varphi_{\omega_2(\eps)}-u).
\]
\item[(iii)] It follows from (i) by writing 
\[
u\ast\varphi_{1,\omega(\eps)}-u\ast\varphi_{2,\omega(\eps)}=(u\ast\varphi_{1,\omega(\eps)}-u)-(u\ast\varphi_{2,\omega(\eps)}-u).
\]
\item[(iv)] This statement is obtained by combining Proposition \ref{prop_reg_nets}(iii) with Remark \ref{rem_net}.
\item[(v)] This is obtained from (iv) by writing $f\ast\varphi_{\omega_1(\eps)}-f\ast\varphi_{\omega_2(\eps)}$ as 
\[
(f\ast\varphi_{\omega_1(\eps)}-f)-(f\ast\varphi_{\omega_2(\eps)}-f).
\]
\item[(vi)] This is obtained from (iv) by writing $(f\ast\varphi_{1,\omega(\eps)}-f\ast\varphi_{2,\omega(\eps)})_\eps$ as
\[
(f\ast\varphi_{1,\omega(\eps)}-f)-(f\ast\varphi_{2,\omega(\eps)}-f).
\]
\end{itemize}
\end{proof}
In an analogous way, by replacing $C^\infty(\R^n)$ with $C^\infty([0,T]\times\R^n)$, we can state the following definition. 
\begin{definition}
\label{def_mod_neg_u}
\leavevmode
\begin{itemize}
\item[(i)] A net $(v_\eps)_\eps\in C^\infty([0,T]\times\R^n)^{(0,1]}$ is moderate if for all $K\Subset\R^n$, $k\in\N$ and $\alpha\in\N^n$ there exist $N\in\N$ and $c>0$ such that
\beq
\label{mod_C_inf_u}
|\partial_t^k\partial^\alpha v_\eps(t,x)|\le c\eps^{-N},
\eeq
uniformly in $t\in[0,T]$, $x\in K$ and $\eps\in(0,1]$.
\item[(ii)] A net $(v_\eps)_\eps\in  C^\infty([0,T]\times\R^n)^{(0,1]}$ is negligible if for all $K\Subset\R^n$, $k\in\N$, $\alpha\in\N^n$ and for all $q\in\N$ there exists $c>0$ such that
\beq
\label{neg_C_inf_u}
|\partial_t^k\partial^\alpha v_\eps(t,x)|\le c\eps^{q},
\eeq
uniformly in $t\in[0,T]$, $x\in K$ and $\eps\in(0,1]$.
\end{itemize}
\end{definition}
Arguing as in Proposition \ref{prop_nets_asymp} we have the following.
\begin{proposition}
\label{prop_nets_asymp_2}
Let $f\in C^\infty([0,T], \E'(\R^n))$. Let $\varphi$ be a mollifier and $\omega(\eps)$ be a positive scale. Then, the net
\[
f_\eps(t,x)=f(t,\cdot)\ast \varphi_{\omega(\eps)}(x)
\]
is moderate. A change of mollifier and/or positive scale does not affect the asymptotic behaviour of the regularisation, i.e., 
\[
f(t,\cdot)\ast \varphi_{1,\omega_1(\eps)}(x)-f(t,\cdot)\ast \varphi_{2,\omega_2(\eps)}(x) \to 0
\]
in $C^\infty([0,T],\E'(\R^n))$.
\end{proposition}
We can now introduce a notion of a `very weak solution' for the Cauchy problem \eqref{CP_wave}. This is similar to the one introduced in \cite{GR:14} but here we measure moderateness and negligibility in terms of $C^\infty$-seminorms rather than in terms of Gevrey-seminorms. We work under the following set of hypotheses:
\begin{itemize}
\item $a_i\in\E'(\R^n)$ with $a_i\ge 0$ for all $i=1,\dots ,n$,
\item $f\in C^\infty([0,T],\E'(\R^n))$,
\item $g_0,g_1\in\E'(\R^n)$.
\end{itemize}

\begin{definition}
\label{def_vws}
The net $(u_\eps)_\eps\in C^\infty([0,T]\times\R^n)^{(0,1]}$ is a very weak solution of the 
Cauchy problem \eqref{CP_wave} if there exist 
\begin{itemize}
\item[(i)] moderate regularisations $(a_{i,\eps})_\eps$ of the coefficients $a_i$ for $i=1,\dots,n$,
\item[(ii)] moderate regularisation $(f_\eps)_\eps$ of the right-hand side $f$,
\item[(iii)] moderate regularisations $(g_{0,\eps})_\eps$ and $(g_{1,\eps})_\eps$ of the initial data $g_0$ and $g_1$, 
respectively,
\end{itemize}
such that $(u_\eps)_\eps$ solves the regularised problem
\[
\begin{split}
\partial_t^2u-\sum_{i=1}^na_{i,\eps}(x)\partial^2_{x_i} u&=f_\eps(t,x),\quad t\in[0,T], x\in\R^n,\\
u(0,x)&=g_{0,\eps},\\
\partial_tu(0,x)&=g_{1,\eps},
\end{split}
\]
for all $\eps\in(0,1]$, and is moderate.
\end{definition}
\begin{definition}
\label{vw_well_posed}
We say that the Cauchy Problem \eqref{CP_wave} is \emph{very weakly well-posed} if a very weak solution exists and it is unique modulo negligible nets, i.e., negligible changes in the regularisation of the equation coefficients and initial data lead to negligible changes in the corresponding very weak solution. \end{definition}
Note that proving the existence of a very weak solution means to prove that there exist suitable mollifiers and $\eps$-scale such that the regularised problem has a moderate solution $(u_\eps)_\eps$. Proving that the Cauchy problem is very weakly well-posed is e\-qui\-va\-lent to prove uniqueness of the solution in the Colombeau algebra $\mathcal{G}([0,T]\times\R^n)$. 

We conclude this section of preliminaries by recalling the proof of the Glaeser's inequality that we will employ in the paper.

\subsection{Glaeser's inequality}
\begin{proposition}[Glaeser's inequality]
\label{prop_Glaeser}
If $a\in C^2(\R^n)$, $a(x)\ge 0$ for all $x\in\R^n$ and 
\[
\sum_{i=1}^n\Vert \partial^2_{x_i}a\Vert_{L^\infty}\le M,
\]
for some constant $M>0$. Then,  
\[
|\partial_{x_i}a(x)|^2\le 2M a(x),
\]
for all $i=1,\dots,n$ and $x\in\R^n$.
\end{proposition}

 \begin{proof}
 It is not restrictive to assume that $i=1$. By Taylor expansion of $a$ we have that 
 \[
 0\le a(x_1+h, x_2,\dots,x_n)=a(x)+\partial_{x_1} a(x)h+\frac{1}{2}\partial^2_1 a(x_1+\theta h,x_2,\dots,x_n)h^2,
 \]
 where $\theta\in[0,1]$. It follows that 
 \[
 0\le a(x)+\partial_{x_1}a(x)h+M\frac{h^2}{2},
 \]
 for all $x\in\R^n$ and $h\in\R$, and therefore the corresponding discriminant must be $\le 0$, i.e.,
 \[
(\partial_{x_1}a(x))^2-2a(x)M\le 0,
 \]
 for all $x\in\R^2$. We can therefore conclude that  
 \[
|\partial_{x_i} a(x)|^2\le 2Ma(x),
 \]
 for $i=1,\dots,n$, uniformly in $x\in\R^n$.
 \end{proof}

\section{The case $n=1$}
\label{sec_1}
The strategy adopted in this paper to prove the $C^\infty$ well-posedness of the Cauchy problem \eqref{CP_wave} can be easily understood if we assume that the space dimension is $1$. Note that this is the case considered by Spagnolo and Taglialatela in \cite{ST} for a wider class of $m$-order weakly hyperbolic equations. Let us therefore start by studying the Cauchy problem
\beq
\label{CP_1_ex}
\begin{split}
\partial_t^2u-a(x)\partial^2_{x} u&=f(t,x),\quad t\in[0,T], x\in\R,\\
u(0,x)&=g_0,\\
\partial_tu(0,x)&=g_1,
\end{split}
\eeq
where $a\in C^2(\R)$ is positive and with bounded second order derivative ($\Vert a''\Vert_\infty\le M$). We also assume that all the functions involved in the system are real valued.
\subsection{System in $U$}
By using the transformation,
\[
U=(\partial_{x}u,\partial_tu)^{T}.
\]
the Cauchy problem above can be rewritten as
\[
\begin{split}
\partial_t U=&A \partial_x U+F,\\
U(0,x)&=(g_0',g_1)^{T},
\end{split}
\]
where
\[
A=\left(
	\begin{array}{cc}
	0& 1\\
        a & 0 
	\end{array}
	\right)
	\]
and
\[
	F=\left(
	\begin{array}{c}
	0 \\
	f
        \end{array}
	\right).
	\]
The matrix $A$, which is in Sylvester form, has the symmetriser 
\[
Q=\left(
	\begin{array}{cc}
	a& 0\\
          0 & 1 
	\end{array}
	\right),
\]
i.e., $QA=A^\ast Q=A^tQ$. The symmetriser $Q$ gives us the energy for the system $\partial_t U=A\partial_xU+F$. Let us define
\[
E(t)=(QU,U)_{L^2}.
\]
By direct computations we have that 
\[
E(t)=(aU_1,U_1)_{L^2}+\Vert U_2\Vert^2_{L^2}
\]
and, being $a\ge 0$ the bound from below
\[
\Vert U_2\Vert^2_{L^2}\le E(t)
\]
holds, for all $t\in[0,T]$. Assume that the initial data $g_0, g_1$ are compactly supported and that $f$ is compactly supported with respect to $x$. By finite speed of propagation, it follows that that the solution $U$ is compactly supported with respect to $x$ as well. Hence, by integration by parts we easily obtain the following energy estimate:
 \[
 \begin{split}
 &\frac{dE(t)}{dt}=(\partial_t(QU),U)_{L^2}+(QU,\partial_tU)_{L^2}\\
 &=(Q\partial_tU,U)_{L^2}+(QU, A\partial_{x}U)_{L^2}+(QU,F)_{L^2}\\
 &=(QA\partial_{x}U, U)_{L^2}+(QU, A\partial_{x}U)_{L^2}+2(QU,F)_{L^2}\\
 &= (QA\partial_{x}U,U)_{L^2}+(A^\ast QU,\partial_{x}U)_{L^2}+2(QU,F)_{L^2}\\
 &=(QA\partial_{x}U,U)_{L^2}+(QAU,\partial_{x}U)_{L^2}+2(QU,F)_{L^2}\\
 &=(QA\partial_{x}U,U)_{L^2}-(\partial_{x}(QAU),U)_{L^2}+2(QU,F)_{L^2}\\
 &=-((QA)'U,U)_{L^2}+2(QU,F)_{L^2}.
\end{split}
 \]
Since 
\[
(QA)'=\left(
	\begin{array}{cc}
	0& a'\\
          a' & 0
	\end{array}
	\right)
\]
by Glaeser's inequality ($|a'(x)|^2\le 2M a(x)$) we immediately have
\begin{multline*}
((QA)'U,U)_{L^2}=2(a'U_1,U_2)_{L^2}\le 2\Vert a'U_1\Vert_{L^2}\Vert U_2\Vert_{L^2}\le \Vert a'U_1\Vert_{L^2}^2+\Vert U_2\Vert^2_{L^2}\\
\le 2M(aU_1,U_1)_{L^2}+\Vert U_2\Vert^2_{L^2}
\le \max(2M,1)E(t).
\end{multline*}
Analogously,
\[
2(QU,F)_{L^2}=2(U_2,f)_{L^2}\le E(t)+\Vert f\Vert_{L^2}.
\]
Hence,
\[
\frac{dE}{dt}\le \max(2M,2)E(t)+\Vert f\Vert_{L^2}.
\]
By Gr\"onwall's lemma and the bound from below for the energy we obtain the following estimate for the entry $U_2$:  
\beq
\label{est_U_2}
\Vert U_{2}(t)\Vert_{L^2}^2\le E(t)\le \biggl(E(0)+\int_{0}^t\Vert f(s)\Vert_{L^2}^2\, ds\biggr){\rm e}^{ct},
\eeq
where $c=\max(2M,2)$. In addition,
\[
\begin{split}
\biggl(E(0)+\int_{0}^t\Vert f(s)\Vert_{L^2}^2\, ds\biggr){\rm e}^{ct}
&\le {\rm e}^{cT}\Vert a\Vert_{\infty}\Vert U_1(0)\Vert_{L^2}^2+ {\rm e}^{cT}\Vert U_2(0)\Vert_{L^2}^2 +{\rm e}^{cT}\int_{0}^t\Vert f(s)\Vert_{L^2}^2\, ds\\
&\le C_{2}\biggl(\Vert g_0\Vert_{H^1}^2+\Vert g_1\Vert_{L^2}^2+\int_{0}^t\Vert f(s)\Vert_{L^2}^2\, ds\biggr),
\end{split}
\]
for all $t\in[0,T]$. Note that the constant $C_2$ depends on $\Vert a\Vert_\infty$ and exponentially on $\Vert a''\Vert_\infty.$ Indeed, setting $M=\Vert a''\Vert_\infty$,
\[
C_2= {\rm e}^{cT}\max(\Vert a\Vert_{\infty},1)= {\rm e}^{\max(2\Vert a''\Vert_\infty,2)T}\max(\Vert a\Vert_{\infty},1).
\]
Concluding,
\[
\Vert U_{2}(t)\Vert_{L^2}^2\le C_{2}\biggl(\Vert g_0\Vert_{H^1}^2+\Vert g_1\Vert_{L^2}^2+\int_{0}^t\Vert f(s)\Vert_{L^2}^2\, ds\biggr).
\]
\subsection{System in V}
We want now to obtain a similar estimate for $U_1$. To do so we transform once more the system by deriving with respect to $x$. Let $V=(\partial_{x}U_1, \partial_{x}U_2)^T$. If we get an estimate for $V_2$ we automatically also get it for $U_1$ since $V_2=\partial_t U_1$. Indeed, it will be enough to apply the fundamental theorem of calculus and make use of the initial conditions. Hence, if $U$ solves
\[
\begin{split}
\partial_t U=&A \partial_x U+F,\\
U(0,x)&=(g'_0,g_1)^{T},
\end{split}
\]
then $V$ solves
\[
\begin{split}
\partial_t V=&A \partial_x V+A'V+F_x\\
V(0,x)&=(g_0'',g_1')^{T}.
\end{split}
\]
The system in $V$ has still $A$ as a principal part matrix but it also has lower order terms. It follows that we can still use the symmetriser $Q$ to define the energy. Hence, 
\[
E(t)=(aV_1,V_1)_{L^2}+\Vert V_2\Vert^2_{L^2}
\]
and
 \[
 \begin{split}
 &\frac{dE(t)}{dt}=(\partial_t(QV),V)_{L^2}+(QV,\partial_tV)_{L^2}\\
 &=(Q\partial_tV,V)_{L^2}+(QV, A\partial_{x}V+A'V)_{L^2}+(QV,F_x)_{L^2}\\
 &=(QA\partial_{x}V+QA'V+QF_x, V)_{L^2}+(QV, A\partial_{x}V+A'V)_{L^2}+(QV,F_x)_{L^2}\\
 &= (QA\partial_{x}V,V)_{L^2}+(A^\ast QV,\partial_{x}V)_{L^2}+(QA'V, V)_{L^2}+(QV, A'V)_{L^2}+2(QV,F_x)_{L^2}\\
 &=(QA\partial_{x}V,V)_{L^2}+(QAV,\partial_{x}V)_{L^2}+2(QA'V, V)_{L^2}+2(QV,F_x)_{L^2}\\
 &=-((QA)'V,V)_{L^2}+2(QA'V, V)_{L^2}+2(QV,F_x)_{L^2}.
\end{split}
 \]
By direct computations
 \[
 \begin{split}
 ((QA)'V,V)_{L^2}&=2(a'V_1,V_2)_{L^2},\\
 2(QA'V, V)_{L^2}&=2(a'V_1,V_2)_{L^2}.
 \end{split}
 \]
 Hence,
 \[
 \frac{dE(t)}{dt}=2(QV,F_x)_{L^2}=2(V_2,f_x)\le E(t)+\Vert f_x\Vert_{L^2}^2,
 \]
and
\begin{multline}
\label{est_V_2}
\Vert V_{2}(t)\Vert_{L^2}^2\le E(t)\le \biggl(E(0)+\int_{0}^t\Vert f_x(s)\Vert_{L^2}^2\, ds\biggr)\\
\le \Vert a\Vert_\infty \Vert V_1(0)\Vert_{L^2}^2+\Vert V_2(0)\Vert_{L^2}^2+\int_{0}^t\Vert f_x(s)\Vert_{L^2}^2\, ds\\
\le \max(\Vert a\Vert_\infty ,1)\biggl(\Vert g_0\Vert_{H^2}^2+\Vert g_1\Vert_{H^1}^2+\int_{0}^t\Vert f(s)\Vert_{H^1}^2\, ds\bigg),
\end{multline}
for all $t\in[0,T]$. Now, let us write $V_2$ as $\partial_tU_1$. By the fundamental theorem of calculus we have
\[
  \Vert U_1(t)\Vert^2_{L^2}\le 2\Vert U_1(t)-U_1(0)\Vert_{L^2}^2+2\Vert U_1(0)\Vert_{L^2}^2
 =2\Big\Vert \int_{0}^tV_{2}\, ds\Big\Vert_{L^2}^2+2\Vert U_1(0)\Vert_{L^2}^2\\
 \]
  By Minkowski's integral inequality
  \[
  \Big\Vert \int_{0}^tV_{2}\, ds\Big\Vert_{L^2}\le \int_{0}^t \Vert V_2(s)\Vert_{L^2}ds
  \]
  and therefore
  \[
    \Vert U_1(t)\Vert^2_{L^2}\le 2\biggl(\int_{0}^t \Vert V_{2}(s)\Vert_{L^2}ds\biggr)^2+2\Vert U_1(0)\Vert_{L^2}^2\le 2t^2 \sup_{s\in[0,t]} \Vert V_{2}(s)\Vert^2_{L^2}+2\Vert U_1(0)\Vert_{L^2}^2.
  \]
By the estimate \eqref{est_V_2} we have
\[
\begin{split}
 \Vert U_1(t)\Vert^2_{L^2}&\le 2t^2\max(\Vert a\Vert_\infty ,1)\biggl(\Vert g_0\Vert_{H^2}^2+\Vert g_1\Vert_{H^1}^2+\int_{0}^t\Vert f(s)\Vert_{H^1}^2\, ds\bigg)+2\Vert U_1(0)\Vert_{L^2}^2\\
& \le 2t^2\max(\Vert a\Vert_\infty ,1)\biggl(\Vert g_0\Vert_{H^2}^2+\Vert g_1\Vert_{H^1}^2+\int_{0}^t\Vert f(s)\Vert_{H^1}^2\, ds\bigg)+2\Vert g_0\Vert_{H^1}^2\\
&\le \max\{2T^2,2\}\max(\Vert a\Vert_\infty ,1)\biggl(\Vert g_0\Vert_{H^2}^2+\Vert g_1\Vert_{H^1}^2+\int_{0}^t\Vert f(s)\Vert_{H^1}^2\, ds\bigg),
 \end{split}
\]
for all $t\in[0,T]$. 
\subsection{System in W}
Analogously, if we want to estimate the $L^2$-norm of $V_1$ we need to repeat the same procedure, i.e.,  to derive the system in $V$ with respect to $x$ and introduce $W=(\partial_{x}V_1, \partial_{x}V_2)^T$. We have that if $V$ solves 
\[
\begin{split}
\partial_t V=&A \partial_x V+A'V+F_x,\\
V(0,x)&=(g_0'',g_1')^{T}.
\end{split}
\]
then $W$ solves
\[
\begin{split}
\partial_t W=&A \partial_x W+2A'W+A''V+F_{xx},\\
W(0,x)&=(g_0''',g_1'')^{T}.
\end{split}
\]
Again, by employing the energy $E(t)=(QW, W)_{L^2}$ we have
 \[
 \begin{split}
 &\frac{dE(t)}{dt}=(\partial_t(QW),W)_{L^2}+(QW\partial_tW)_{L^2}\\
 &=(Q\partial_tW,W)_{L^2}+(QW, A\partial_{x}W+2A'W)_{L^2}+(QW,A''V+F_{xx})_{L^2}\\
 &=(QA\partial_{x}W+2QA'W+QA''V+Q F_{xx}, W)_{L^2}+(QW, A\partial_{x}W+2A'W)_{L^2}\\
 &+(QW,A''V+ F_{xx})_{L^2}\\
 &=-((QA)'W,W)_{L^2}+(2QA'W, W)_{L^2}+(QW,2A'W)_{L^2}+2(QW,A''V+ F_{xx})_{L^2}\\
 &=-((QA)'W,W)_{L^2}+4(QA'W, W)_{L^2}+2(QW,A''V+ F_{xx})_{L^2}\\
 &=-2(a'W_1,W_2)_{L^2}+4(a'W_1,W_2)_{L^2}+2(QW, A''V)_{L^2}+2(QW,  F_{xx})_{L^2}\\
 &=2(a'W_1,W_2)_{L^2}+2(W_2, a''V_1)_{L^2}+2(W_2,  F_{xx})_{L^2}.
\end{split}
 \]
 Arguing as above we get
 \beq
 \label{est_E_W_1}
 \frac{dE(t)}{dt}\le \max(2M,2)E(t)+\Vert  F_{xx}\Vert_{L^2}+2(W_2, a''V_1)_{L^2}.
 \eeq
 We need to estimate the term $(W_2, a''V_1)_{L^2}$. By Cauchy-Schwarz we have
 \beq
  \label{est_E_W_2}
 2(W_2, a''V_1)_{L^2}\le \Vert W_2\Vert_{L^2}^2+\Vert a''\Vert_\infty^2\Vert V_1\Vert^2_{L^2}
 \eeq
and since $\partial_t V_1=W_2$ we can write
\begin{multline}
 \label{est_E_W_3}
\Vert V_1\Vert^2_{L^2}=\biggl\Vert \int_0^t \partial_tV_1(s)\, ds+V_1(0)\biggr\Vert^2\le 2\biggl\Vert \int_0^t \partial_tV_1(s)\, ds\biggr\Vert^2_{L^2}+2\Vert V_1(0)\Vert_{L^2}^2\\
\le2\biggl(\int_0^t\Vert W_2(s)\Vert_{L^2}\, ds\biggr)^2+2\Vert V_1(0)\Vert_{L^2}^2
\le 2t^2\int_0^t\Vert W_2(s)\Vert^2_{L^2}\, ds+2\Vert V_1(0)\Vert_{L^2}^2\\
\le 2T^2\int_0^t E(s)\, ds+2\Vert g_0\Vert_{H^2}^2.
\end{multline}
Combining, \eqref{est_E_W_1} with \eqref{est_E_W_2} and  \eqref{est_E_W_3} we obtain
\[
\begin{split}
 \frac{dE(t)}{dt}&\le \max(2M,2)E(t)+\Vert  F_{xx}\Vert_{L^2}+2(W_2, a''V_1)_{L^2}\\
 &\le \max(2M,2)E(t)+\Vert  F_{xx}\Vert_{L^2}+\Vert W_2\Vert_{L^2}^2+\Vert a''\Vert_\infty^2\Vert V_1\Vert^2_{L^2}\\
 &\le  \max(2M,2)E(t)+\Vert  F_{xx}\Vert_{L^2}+\Vert W_2\Vert_{L^2}^2+2T^2\Vert a''\Vert_\infty^2\int_0^t E(s)\, ds+2\Vert a''\Vert_\infty^2\Vert g_0\Vert_{H^2}^2\\
 &\le c'(M, \Vert a''\Vert_\infty^2, T)\biggl(E(t)+\int_0^t E(s)\, ds+\Vert f(s)\Vert_{H^2}^2+\Vert g_0\Vert_{H^2}^2\biggr).
\end{split}
\]
By a simple Gr\"onwall type lemma (see Lemma 6.2 in \cite{ST} or Lemma \ref{lem_ST_G}) we conclude that there exists a constant $C_2'(c',T, \Vert a\Vert_\infty)>0$ such that 
\[
\Vert W_2\Vert_{L^2}^2\le E(t)\le C_2'\biggl(\Vert g_0\Vert_{H^3}^2+\Vert g_1\Vert_{H^2}^2+\int_{0}^t\Vert f(s)\Vert_{H^2}^2\, ds\biggr).
\]
Note that the constant $C_2'$ depends exponentially on $\Vert a''\Vert_{\infty}^2$, $M$ and $T$ and linearly on $\Vert a\Vert_\infty$.
Since $W_2=\partial_t V_1$ by the fundamental theorem of calculus the estimate above can be stated (with a different constant $C_2'$) for $V_1$. 
\subsection{Sobolev estimates} Summarising, we have proven that if $U$ is a solution of the Cauchy problem 
\[
\begin{split}
\partial_t U=&A \partial_x U+F,\\
U(0,x)&=(\partial_{x}g_0,g_1)^{T},
\end{split}
\]
then
\[
\begin{split}
\Vert U_1(t)\Vert_{L^2}^2&\le C_1(T,\Vert a\Vert_\infty)\biggl(\Vert g_0\Vert_{H^2}^2+\Vert g_1\Vert_{H^1}^2+\int_{0}^t\Vert f(s)\Vert_{H^1}^2\, ds\bigg),\\
\Vert U_{2}(t)\Vert_{L^2}^2&\le C_{2}(T,M,\Vert a\Vert_\infty)\biggl(\Vert g_0\Vert_{H^1}^2+\Vert g_1\Vert_{L^2}^2+\int_{0}^t\Vert f(s)\Vert_{L^2}^2\, ds\biggr).\\
\end{split}
\]
It follows that 
\[
\Vert U(t)\Vert_{L^2}^2\le C_0\biggl(\Vert g_0\Vert_{H^2}^2+\Vert g_1\Vert_{H^1}^2+\int_{0}^t\Vert f(s)\Vert_{H^1}^2\, ds\biggr).
\]
Passing now to the system in $V$ we have proven that 
\[
\begin{split}
\Vert V_1(t)\Vert_{L^2}^2&\le C(T,M,M^2, \Vert a\Vert_\infty)\biggl(\Vert g_0\Vert_{H^3}^2+\Vert g_1\Vert_{H^2}^2+\int_{0}^t\Vert f(s)\Vert_{H^2}^2\, ds\biggr),\\
\Vert V_{2}(t)\Vert_{L^2}^2&\le \max(\Vert a\Vert_\infty ,1)\biggl(\Vert g_0\Vert_{H^2}^2+\Vert g_1\Vert_{H^1}^2+\int_{0}^t\Vert f(s)\Vert_{H^1}^2\, ds\biggr),\\
\end{split}
\]
and therefore there exists a constant $C_1>0$, with dependence on $T$, $M$, $M^2$ and $\Vert a\Vert_\infty$ as above, such that
\[
\Vert U(t)\Vert_{H^1}^2\le C_1\biggl(\Vert g_0\Vert_{H^3}^2+\Vert g_1\Vert_{H^2}^2+\int_{0}^t\Vert f(s)\Vert_{H^2}^2\, ds\biggr).
\]
This immediately gives the estimates
\[
\Vert u(t)\Vert_{H^{k+1}}^2\le C_k\biggl(\Vert g_0\Vert_{H^{k+2}}^2+\Vert g_1\Vert_{H^{k+1}}^2+\int_0^t \Vert f(s)\Vert_{H^{k+1}}^2\, ds\biggl),
\]
for all $t\in[0,T]$ and $k=-1,0,1$, where $C_k$ depends on $T$, $M$, $M^{k+1}$ and $\Vert a\Vert_\infty$.

\subsection{Conclusion} We have proven $H^k$-Sobolev well-posedness for the Cauchy problem \eqref{CP_1_ex} for $k=0,1,2$, provided that $a\ge 0$ is of class $C^2$ with bounded derivatives up to order $2$. Existence of the solution is obtained via a standard perturbation argument on the strictly hyperbolic case (see \cite{ST} and the proof of Theorem \ref{theo_main}) and the uniqueness follows from the estimates above. Clearly, one can iterate this argument and obtain Sobolev estimates for every order $k$. The iteration will involve further derivatives of the coefficient $a$, namely up to order $k+1$ and therefore to get well-posedness in every Sobolev space we require that $a$ is smooth and has bounded derivatives of any order.

\section{$C^\infty$ well-posedness in arbitrary space dimension}
\label{sec_CP_smooth} 
In this section we study the Cauchy problem \eqref{CP_wave} 
\[
\begin{split}
\partial_t^2u-\sum_{i=1}^na_i(x)\partial^2_{x_i} u&=f(t,x),\quad t\in[0,T], x\in\R^n,\\
u(0,x)&=g_0,\\
\partial_tu(0,x)&=g_1,
\end{split}
\]
when the coefficients $a_i$ are regular and positive on $\R^n$, i.e., $a_i\in B^\infty([0,T]\times\R^n)$ and $a_i\ge 0$ for all $i=1,\dots,n$.  All the functions above are assumed to be real-valued. We give an alternative proof of the $C^\infty$ well-posedness result obtained in \cite{O70}. Our strategy consists in transforming \eqref{CP_wave} into a system of first order differential equations and then to employ the symmetriser of this system in order to get energy and energy estimates. Note that since we want to employ integration by parts it is convenient to work with differential operators rather than pseudo-differential operators. This means that we cannot use the standard reduction of a second order scalar differential equation into a $2\times 2$ system of pseudodifferential equations as in \cite{GR:11, GR:12, GR:14}. As a consequence, the size of the our system matrix will not be $2$ but it will depend on the spatial dimension $n$.
\subsection{Reduction into a system and construction of the symmetriser}
We transform the Cauchy problem above into a first order system of differential equations of size $n+1$. Let
\[
U=(\partial_{x_1}u,\partial_{x_2}u,\cdots, \partial_{x_n}u, \partial_tu)^{T}.
\]
The Cauchy problem above is equivalent to the $n+1\times n+1$ system
\beq
\label{CP1_syst_inh}
\begin{split}
\partial_tU&=\sum_{k=1}^nA_k(x)\partial_{x_k} U+F,\\
U(0,x)&=(\partial_{x_1}g_0,\partial_{x_2}g_0,\cdots, \partial_{x_n}g_0, g_1)^{T},
\end{split}
\eeq
where, $A_k=(a_{k,ij})_{ij}$ with
\[
\begin{split}
a_{k,ij}&=1,\quad\text{for $i=k$ and $j=n+1$},\\
a_{k,ij}&=a_k,\quad\text{for $i=n+1$ and $j=k$},\\
a_{k,ij}&=0,\quad\text{otherwise},
\end{split}
\]
and
	\[
	F=\left(
	\begin{array}{c}
	0 \\
	0 \\
 
        \vdots\\
        0\\
      f
        \end{array}
	\right).
	\]
Note that each matrix $A_k$ can be regarded as a matrix in Sylvester form with $n-1$ rows and columns identically zero. Indeed, when $n=2$ we have
\[
A_1=\left(
	\begin{array}{ccc}
	0& 0 & 1\\
	0& 0 & 0\\ 
	a_1 & 0 & 0
	\end{array}
	\right)
\]
and
\[
A_2=\left(
	\begin{array}{ccc}
	0& 0 & 0\\
	0& 0 & 1\\ 
	0 & a_2 & 0
	\end{array}
	\right).
\]
and when $n=3$ we have
\[
A_1=\left(
	\begin{array}{cccc}
	0& 0 & 0 & 1\\
	0& 0 & 0 & 0\\
	0& 0 & 0 & 0\\
	a_1 & 0 & 0 & 0
	\end{array}
	\right),\quad A_2=\left(
	\begin{array}{cccc}
	0& 0 & 0 & 0\\
	0& 0 & 0 & 1\\
	0& 0 & 0 & 0\\
	0 & a_2 & 0 & 0
	\end{array}
	\right), \quad A_3=\left(
	\begin{array}{cccc}
	0& 0 & 0 & 0\\
	0& 0 & 0 & 0\\
	0 & 0 & 0 & 1\\
	0 & 0 & a_3 & 0
	\end{array}
	\right).
\]
It is possible to construct a common $Q=Q_n$ symmetriser for all the matrices $A_k$, $k=1,\dots,n$. For the sake of simplicity, in the sequel we will omit the subscript $n$ in the notation $Q_n$ of the symmetriser. 
\begin{proposition}
\label{prop_sym_wave}
\leavevmode
\begin{itemize}
\item[(i)] The diagonal $n+1\times n+1$ matrix
\[
Q=\left(
	\begin{array}{cccc}
	a_1& 0 & \cdots & 0\\
	0& a_2 & \cdots & 0\\
	\vdots& \vdots & \vdots & \vdots\\
	\cdots& \cdots & a_n & 0\\
	0 & \cdots &0 & 1
	\end{array}
	\right)
\]
is a symmetriser for every matrix $A_k$ with $k=1,\dots, n$, i.e. $QA_k=A_k^\ast Q$. Moreover, $QA_k$ has $ij$-entry and $ji$-entry equal to $a_k$ for $i=k$, $j=n+1$ and it is identically zero otherwise.  
\item[(ii)] If $a_k\ge 0$ for all $k=1,\dots,n$ then $\lara{Qv,v}\ge |v_{n+1}|^2$ for all $v\in\R^{n+1}$.
\end{itemize}
\end{proposition}

\begin{proof}

(i) The $ih$-entry of the product $QA_k$ is given by $\sum_{j=1}^{n+1}q_{ij}a_{k,jh}$. Since $Q$ is diagonal it follows that this sum can be reduced to $q_{ii}a_{k,ih}$. Hence the only non-zero entries are obtained for $i=n+1$ and $h=k$ and for $i=k$ and $h=n+1$. In both cases we obtain that the entry of the product is equal to $a_k$.\\
(ii) Since the coefficients $a_k$ are non-negative it follows by direct computations that $\lara{Qv,v}=\sum_{k=1}^n a_kv_k^2+|v_{n+1}|^2\ge |v_{n+1}|^2,$ for all $v\in\R^{n+1}$.
\end{proof}
We define the energy 
\beq
\label{EU1}
E(t)=(QU,U)_{L^2}=\sum_{i=1}^n (a_iU_i,U_i)_{L^2}+\Vert U_{n+1}\Vert^2_{L^2},
\eeq
where $L^2=L^2(\R^n_x)$. For the sake of simplicity we keep writing $L^2$ even when we are considering its $n$-product like in $(QU,U)_{L^2}$. It follows immediately from Proposition \ref{prop_sym_wave}(ii) that 
\beq
\label{bfb}
\Vert U_{n+1}\Vert_{L^2}^2\le E(t)\le \max_{i=1,\dots,n}\{\Vert a_i\Vert_{L^\infty}\}\sum_{i=1}^n\Vert U_i\Vert_{L^2}^2+\Vert U_{n+1}\Vert^2_{L^2}.
\eeq
Note that since the equation generating our system is weakly hyperbolic we cannot bound the energy from below with the $L^2$-norm of $U$ but with the $L^2$-norm of its last component.
Our plan is to estimate the energy $E(t)$ and prove in this way that our Cauchy problem is $C^\infty$ well-posed (or equivalently well-posed in every Sobolev space). We start by proving $L^2$-estimate for $U$ and we then pass to any Sobolev order.

\subsection{$L^2$-estimates of $U$.}
We will first focus on the component $U_{n+1}$ and we will then pass to consider $U_i$ with $i=1,\dots, n$.

\subsubsection{$L^2$-estimates of $U_{n+1}$}
In the sequel we use the fact that our system is of dif\-fe\-rential operators rather than pseudo-differential operators and we apply integration by parts. We assume that $g_0$, $g_1$ and $f(t,\cdot)$ are compactly supported so by finite speed propagation we can assume that $U=U(t,\cdot)$ is compactly supported as well. Hence,
 \[
 \begin{split}
 &\frac{dE(t)}{dt}=(\partial_t(QU),U)_{L^2}+(QU,\partial_tU)_{L^2}\\
 &=(Q\partial_tU,U)_{L^2}+(QU, \sum_{k=1}^nA_k\partial_{x_k}U)_{L^2}+(QU,F)_{L^2}\\
 &=(Q\sum_{k=1}^nA_k\partial_{x_k}U, U)_{L^2}+(QU, \sum_{k=1}^nA_k\partial_{x_k}U)_{L^2}+2(QU,F)_{L^2}\\
 &= \sum_{k=1}^n \biggl((QA_k\partial_{x_k}U,U)_{L^2}+(A_k^\ast QU,\partial_{x_k}U)_{L^2}\biggr)+2(QU,F)_{L^2}\\
 &=\sum_{k=1}^n \biggl((QA_k\partial_{x_k}U,U)_{L^2}+(QA_kU,\partial_{x_k}U)_{L^2}\biggr)+2(QU,F)_{L^2}\\
 &=\sum_{k=1}^n \biggl((QA_k\partial_{x_k}U,U)_{L^2}-(\partial_{x_k}(QA_kU),U)_{L^2}\biggr)+2(QU,F)_{L^2}\\
 &=-\sum_{k=1}^n(\partial_{x_k}(QA_k)U,U)_{L^2}+2(QU,F)_{L^2}.
\end{split}
 \]
 Our aim is to estimate $-\sum_{k=1}^n(\partial_{x_k}(QA_k)U,U)_{L^2}+2(QU,F)_{L^2}$ with the energy $E(t)$. This is possible thanks to the Glaeser's inequality (Proposition \ref{prop_Glaeser}) and the fact that, by direct computations,
\beq
\label{formula_E}
\begin{split}
(\partial_{x_k}(QA_k)U,U)_{L^2}&=2(\partial_{x_k}a_k U_k, U_{n+1})_{L^2},\\
 (QU,F)_{L^2}&=(U_{n+1},f)_{L^2}.
 \end{split}
\eeq
\begin{proposition}
\label{prop_Energy_1}
Let 
\[
E(t)=(QU,U)_{L^2}=\sum_{i=1}^n (a_iU_i,U_i)_{L^2}+\Vert U_{n+1}\Vert^2_{L^2}
\]
be the Energy of the system \eqref{CP1_syst_inh}. Assume that
\begin{itemize}
\item[(H1)] the coefficients $a_i$ are bounded and $a_i\ge 0$ for all $i=1,\dots,n$;
\item[(H2)] the coefficients $a_i$ are of class $C^2$ with bounded second order derivatives, i.e., $\exists M>0$ such that
\[
\sum_{j=1}^n\Vert \partial^2_{x_j}a_i\Vert_{L^\infty}\le M,
\]
for all $x\in\R^n$, for all $i=1,\dots,n$.
\end{itemize}
Then, 
\begin{itemize}
\item[(i)] there exists a constant $c=c(n,M)>0$ such that 
\beq
\label{est_dE}
\frac{dE(t)}{dt}\le cE(t)+\Vert f(t)\Vert_{L^2}^2,
\eeq
for all $t\in[0,T]$;
\item[(ii)] There exists a constant $C=C_{n+1}(n,M,\max_{i=1,\dots,n}\Vert a_i\Vert_{\infty},T)>0$ such that 
\beq
\label{est_U_n+1}
\Vert U_{n+1}(t)\Vert_{L^2}^2\le C_{n+1}\biggl(\Vert g_0\Vert_{H^1}^2+\Vert g_1\Vert_{L^2}^2+\int_{0}^t\Vert f(s)\Vert_{L^2}^2\, ds\biggr)
\eeq
for all $t\in[0,T]$.  
\end{itemize}
\end{proposition}
\begin{proof}
From \eqref{formula_E} and Cauchy-Schwarz inequality we have that 
\[
\begin{split}
\frac{dE(t)}{dt}&\le  2\sum_{k=1}^n\Vert \partial_{x_k}a_kU_k\Vert_{L^2}\Vert U_{n+1}\Vert_{L^2}+2\Vert f\Vert_{L^2}\Vert U_{n+1}\Vert_{L^2}\\
&\le\sum_{k=1}^n\Vert \partial_{x_k}a_kU_k\Vert^2_{L^2}+n\Vert U_{n+1}\Vert^2_{L^2}+\Vert f\Vert^2_{L^2}+\Vert U_{n+1}\Vert^2_{L^2}\\
&=\sum_{k=1}^n\Vert \partial_{x_k}a_kU_k\Vert^2_{L^2}+(n+1)\Vert U_{n+1}\Vert^2_{L^2}+\Vert f\Vert^2_{L^2}.
\end{split}
\]
We now write $\Vert \partial_{x_k}a_kU_k\Vert^2_{L^2}$ as
\[
(\partial_{x_k}a_kU_k,\partial_{x_k}a_kU_k)_{L^2}=((\partial_{x_k}a_k)^2U_k,U_k)_{L^2}.
\]
By Glaeser's inequality  ($|\partial_{x_k}a_k(x)|^2\le 2Ma_k(x)$) we obtain the estimate
\[
 \Vert \partial_{x_k}a_kU_k\Vert^2_{L^2}\le 2M(a_kU_k,U_k)_{L^2}. 
\]
Thus
\[
\begin{split}
\frac{dE(t)}{dt}&\le 2M\sum_{k=1}^n(a_kU_k,U_k)_{L^2}+(n+1)\Vert U_{n+1}\Vert^2_{L^2}+\Vert f\Vert^2_{L^2}\\
&\le \max\{2M, n+1\} E(t)+\Vert f\Vert^2_{L^2}.
\end{split}
\]
This proves assertion (i) with $c=\max\{2M, n+1\}$. By now combining the bound from below \eqref{bfb} with Gr\"onwall's lemma we get  
\beq
\label{est_U_n+1}
\begin{split}
\Vert U_{n+1}(t)\Vert_{L^2}^2\le E(t)&\le \biggl(E(0)+\int_{0}^t\Vert f(s)\Vert_{L^2}^2\, ds\biggr){\rm e}^{ct}\\
&\le {\rm e}^{cT}\max_{i=1,\dots,n}\Vert a_i\Vert_{\infty}\sum_{i=1}^{n+1}\Vert U_i(0)\Vert_{L^2}^2+{\rm e}^{cT}\int_{0}^t\Vert f(s)\Vert_{L^2}^2\, ds\\
&\le C_{n+1}\biggl(\Vert g_0\Vert_{H^1}^2+\Vert g_1\Vert_{L^2}^2+\int_{0}^t\Vert f(s)\Vert_{L^2}^2\, ds\biggr)\
\end{split}
\eeq
for all $t\in[0,T]$, where the existence of $C_{n+1}=C_{n+1}(n,M,\max_{i=1,\dots,n}\Vert a_i\Vert_{\infty},T)>0$ is clear from the estimates above.
\end{proof}
\subsubsection{$L^2$-estimates of $U_i$ when $i\neq n+1$} To be able to estimate the entries $U_i$ of $U$ when $i\neq n+1$ we need to transform the system \eqref{CP1_syst_inh}. We begin by noting that if $U$ solves 
\[
\partial_t U=\sum_{k=1}^nA_k(x)\partial_{x_k} U+F
\]
then, for all $i=1,\dots,n$
\[
\partial_t\partial_{x_i}U=\sum_{k=1}^n A_k(x)\partial_{x_k}\partial_{x_i}U+\sum_{k=1}^n\partial_{x_i}A_k(x)\partial_{x_k} U+\partial_{x_i}F.
\]
Let $V=(\partial_{x_1}U, \partial_{x_2}U,\dots,\partial_{x_n}U)^T$. This is a $n(n+1)$-column vector. We immediately see that if $U$ solves \eqref{CP1_syst_inh} then $V$ solves the Cauchy problem

\beq
\label{CP1_V}
\begin{split}
\partial_t V&=\sum_{k=1}^n\wt{A_k}(x)\partial_{x_k} V+\wt{B}V+\wt{F},\\
V(0,x)&=(\partial_{x_1}U(0,x),\partial_{x_2}U(0,x),\dots,\partial_{x_n}U(0,x))^T,
\end{split}
\eeq
where the matrices involved have size $n(n+1)\times n(n+1)$ and the following structure:
\[
\wt{A_k}=\left(
	\begin{array}{cccc}
	A_k & 0 & \cdots & 0\\
	0 & A_k & \cdots & 0\\
	\vdots & \vdots & \vdots & \vdots\\
	0 & 0 & \cdots & A_k
	\end{array}
	\right),
\]
for $k=1,\dots,n$,
\[
\wt{B}=\left(
	\begin{array}{ccccc}
	\partial_{x_1}A_1 & \partial_{x_1}A_2 & \cdots & \cdots & \partial_{x_1}A_n\\
	\partial_{x_2}A_1& \partial_{x_2}A_2 & \cdots & \cdots & \partial_{x_2}A_n\\
	\vdots & \vdots & \vdots & \vdots & \vdots\\
	\partial_{x_k}A_1&  \cdots &\partial_{x_k}A_k & \cdots & \partial_{x_k}A_n\\
	\vdots & \vdots & \vdots & \vdots & \vdots\\
	\partial_{x_n}A_1& \partial_{x_n}A_2 & \cdots & \cdots & \partial_{x_n}A_n\\
	 
	\end{array}
	\right)
\]
and
\[
\wt{F}=\left(
	\begin{array}{c}
	\partial_{x_1}F\\
	\partial_{x_2}F\\
	\vdots\\
	\partial_{x_n}F\\
	\end{array}
	\right).
\]
Let 
\[
\wt{Q}=\left(
	\begin{array}{ccccc}
	Q & 0 &\cdots & 0 & 0\\
	0 & Q & \cdots & 0 & 0\\
	\vdots & \vdots & \vdots & \vdots & \vdots\\
	0 & 0 & \cdots & Q & 0\\
	0 & 0 &\cdots & \cdots & Q
	\end{array}
	\right),
\]
be a block diagonal matrix with $n$ identical blocks equal to the symmetriser 
\[
Q=\left(
	\begin{array}{cccc}
	a_1& 0 & \cdots & 0\\
	0& a_2 & \cdots & 0\\
	\vdots& \vdots & \vdots & \vdots\\
	\cdots& \cdots & a_n & 0\\
	0 & \cdots &0 & 1
	\end{array}
	\right).
\]
By construction $\wt{Q}$ is a symmetriser of $\wt{A_k}$ for every $k=1,\dots, n$, i.e., $\wt{Q}\wt{A_k}={\wt{A_k}}^\ast \wt{Q}$. We now define the energy
\beq
\label{Energy_V}
\begin{split}
E(t)&=(\wt{Q}V,V)_{L^2}\\
&=\sum_{i=1}^n (a_iV_i,V_i)_{L^2}+\Vert V_{n+1}\Vert_{L^2}^2+\sum_{k=1}^{n-1}\biggl(\sum_{i=1}^n (a_i V_{k(n+1)+i}, V_{k(n+1)+i})_{L^2}+\Vert V_{(k+1)(n+1)}\Vert_{L^2}^2\biggr)\\
&=\sum_{k=0}^{n-1}\biggl(\sum_{i=1}^n (a_i V_{k(n+1)+i}, V_{k(n+1)+i})_{L^2}+\Vert V_{(k+1)(n+1)}\Vert_{L^2}^2\biggr)
\end{split}
\eeq
Moreover,
\beq
\label{Energy_V_est}
\sum_{k=0}^{n-1}\Vert V_{(k+1)(n+1)}\Vert_{L^2}^2\le E(t)\le \sum_{k=0}^{n-1}\biggl(\sum_{i=1}^n \Vert a_i\Vert_{L^\infty}\Vert V_{k(n+1)+i}\Vert_{L^2}^2+\Vert V_{(k+1)(n+1)}\Vert_{L^2}^2\biggr).
\eeq
Note that \eqref{Energy_V_est} implies
\beq
\label{Energy_V_est_2}
\sum_{k=0}^{n-1}\Vert V_{(k+1)(n+1)}\Vert_{L^2}^2\le E(t)\le \max_{i=1,\dots,n}\Vert a_i\Vert_{L^\infty}\sum_{k=0}^{n-1}\sum_{i=1}^n \Vert V_{k(n+1)+i}\Vert_{L^2}^2+\sum_{k=0}^{n-1}\Vert V_{(k+1)(n+1)}\Vert_{L^2}^2.
\eeq
which is analogous to the estimate \eqref{bfb}. 

\begin{example}
An an explanatory example we write down the transformation from the system in $U$ into the system in $V$ when $n=2$. We start from 
\[
\begin{split}
\partial_tU&=A_1(x)\partial_{x_1} U+A_2(x)\partial_{x_2} U+F\\
U(0,x)&=(\partial_{x_1}g_0,\partial_{x_2}g_0,g_1)^{T},
\end{split}
\]
where  
\[
A_1=\left(
	\begin{array}{ccc}
	0& 0 & 1\\
	0& 0 & 0\\ 
	a_1 & 0 & 0
	\end{array}
	\right),
\]
\[
A_2=\left(
	\begin{array}{ccc}
	0& 0 & 0\\
	0& 0 & 1\\ 
	0 & a_2 & 0
	\end{array}
	\right)
\]
and 
\[
F=\left(
	\begin{array}{c}
	0\\
	0\\
	f
	\end{array}
	\right).
\]
Thus, $V=(\partial_{x_1}U,\partial_{x_2}U)^T$ is an element of $\R^6$ which transforms the Cauchy problem above into    
\[
 \partial_t V=\left(
	\begin{array}{cc}
	A_1 & 0\\
	0 & A_1
	\end{array}
	\right)\partial_{x_1}V+\left(
	\begin{array}{cc}
	A_2 & 0\\
	0 & A_2
	\end{array}
	\right)\partial_{x_2}V+ \left(
	\begin{array}{cc}
	\partial_{x_1}A_1 & \partial_{x_1}A_2\\
	\partial_{x_2}A_1 & \partial_{x_2}A_2
	\end{array}
	\right)V+\left(
	\begin{array}{c}
	\partial_{x_1}F\\
	\partial_{x_2}F
	\end{array}
	\right),
\]
with initial data $V(0)=(\partial_{x_1}U(0), \partial_{x_2}U(0))^T$. The matrices involved have size $6\times 6$. Finally, the symmetriser
\[
\wt{Q}=\left(
	\begin{array}{cc}
	Q & 0\\
	0 & Q
	\end{array}
	\right),
\]
generates the energy
\beq
\label{E_V}
\begin{split}
 E(t)&=(\wt{Q}V,V)_{L^2}\\
 &=(a_1V_1,V_1)_{L^2}+(a_2V_2,V_2)_{L^2}+\Vert V_3\Vert_{L^2}^2+(a_1V_4,V_4)_{L^2}+(a_2V_5,V_5)_{L^2}+\Vert V_6\Vert_{L^2}^2.
 \end{split}
 \eeq
 \end{example}
 We now go back to the Cauchy Problem \eqref{CP1_V} and we estimate the $L^2$-norm of the components $V_{(k+1)(n+1)}$ of the vector $V$ for $k=0,\dots,n-1$. The first step is to differentiate and estimate the energy. Arguing as before for the system in $U$ and the corresponding energy, we obtain the following:  
 \beq
\label{energy_V_dt}
\begin{split}
 &\frac{dE(t)}{dt}=(\partial_t \wt{Q}V,V)_{L^2}+(\wt{Q}V,\partial_tV)_{L^2}\\
 &=(\wt{Q}\partial_tV,V)_{L^2}+(\wt{Q}V, \sum_{k=1}^n\wt{A_k}(x)\partial_{x_k} V+\wt{B}V+\wt{F})_{L^2}\\
 &=(\wt{Q}\sum_{k=1}^n\wt{A_k}\partial_{x_k}V,V)_{L^2}+(\wt{Q}V, \sum_{k=1}^n\wt{A_k}(x)\partial_{x_k} V)_{L^2}+(\wt{Q}\wt{B}V,V)_{L^2}+(\wt{Q}V,\wt{B}V)_{L^2}\\
 &+(\wt{Q}\wt{F},V)_{L^2}+(\wt{Q}V,\wt{F})_{L^2}\\
 &=\sum_{k=1}^n (\wt{Q}\wt{A_k}\partial_{x_k}V,V)_{L^2}+\sum_{k=1}^n (\wt{A_k}^\ast\wt{Q}V,\partial_{x_k}V)_{L^2}+((\wt{Q}\wt{B}+\wt{B}^\ast\wt{Q})V,V)_{L^2}+2(\wt{Q}V,\wt{F})_{L^2}\\
 &=-\sum_{k=1}^n(\partial_{x_k}(\wt{A_k}^\ast\wt{Q})V,V)_{L^2}+((\wt{Q}\wt{B}+\wt{B}^\ast\wt{Q})V,V)_{L^2}+2(\wt{Q}V,\wt{F})_{L^2}\\
 &=-\sum_{k=1}^n(\partial_{x_k}(\wt{Q}\wt{A_k})V,V)_{L^2}+((\wt{Q}\wt{B}+\wt{B}^\ast\wt{Q})V,V)_{L^2}+2(\wt{Q}V,\wt{F})_{L^2}
\end{split}
\eeq
Note that formally we have the same kind of mathematical expression obtained for the energy of $U$ but here the matrices involved are different in term of size and entries and lower order terms appear as well.
\begin{proposition}
\label{prop_terms_en_V}
By definition of the matrices $\wt{Q}$, $\wt{A_k}$, $k=1,\dots,n$, $\wt{B}$ and $\wt{F}$ we have that 
\begin{itemize}
\item[(i)] $(\partial_{x_k}(\wt{Q}\wt{A_k})V,V)_{L^2}=2\sum_{j=0}^{n-1} (\partial_{x_k}a_kV_{k+j(n+1)}, V_{(j+1)(n+1)})_{L^2}$,
\item[(ii)] $((\wt{Q}\wt{B}+\wt{B}^\ast\wt{Q})V,V)_{L^2}=2\sum_{k=1}^n\sum_{j=1}^n(\partial_{x_k}a_jV_{j+(j-1)(n+1)},V_{k(n+1)})_{L^2}$,
\item[(iii)] $(\wt{Q}V,\wt{F})_{L^2}=\sum_{k=1}^n (V_{k(n+1)}, \partial_{x_k}f)_{L^2}$.
\end{itemize}
\end{proposition}
\begin{proof}
Since the matrices $\wt{Q}$ and $\wt{A_k}$ are block-diagonal we can argue at the block level. By \eqref{formula_E} we get
\[
\begin{split}
(\partial_{x_k}(\wt{Q}\wt{A_k})V,V)_{L^2}&=2\sum_{j=0}^{n-1}(\partial_{x_k}a_kV_{k+j(n+1)}, V_{(j+1)(n+1)})_{L^2},\\
(\wt{Q}V,\wt{F})_{L^2}&=\sum_{k=1}^n (V_{k(n+1)}, \partial_{x_k}f)_{L^2},
\end{split}
\]
which proves assertions (i) and (iii). By direct computations we have that 
\[
\wt{Q}\wt{B}=\left(
	\begin{array}{cccc}
	Q\partial_{x_1}A_1 & Q\partial_{x_1}A_2 & \cdots & Q\partial_{x_1}A_n \\
	Q\partial_{x_2}A_1& Q\partial_{x_2}A_2 & \cdots & Q\partial_{x_2}A_n \\
	\vdots & \vdots & \vdots & \vdots \\
	Q\partial_{x_k}A_1&  \cdots &Q\partial_{x_k}A_k & \cdots \\
	\vdots & \vdots & \vdots & \vdots \\
	Q\partial_{x_n}A_1&Q \partial_{x_n}A_2 & \cdots & Q\partial_{x_n}A_n\\
	 \end{array}
	\right)
\]
and $((\wt{Q}\wt{B}+\wt{B}^\ast\wt{Q})V,V)_{L^2}=2(\wt{Q}\wt{B}V,V)_{L^2}$. Note that the matrix $Q\partial_{x_k}A_j$ has only the entry $(n+1)j$ different from $0$ and equal to $\partial_{x_k}a_j$. It follows that 
\[
 ((\wt{Q}\wt{B}+\wt{B}^\ast\wt{Q})V,V)_{L^2}\\
=2\sum_{k=1}^n\sum_{j=1}^n(\partial_{x_k}a_jV_{j+(j-1)(n+1)}, V_{k(n+1)})_{L^2}.
\]
\end{proof}
\begin{remark}
Note that when $n=1$ then (i) and (ii) in the previous proposition give $2(a'V_1,V_2)_{L^2}$ as already observed in Section \ref{sec_1} and therefore 
\[
\frac{dE(t)}{dt}= (V_{2}, \partial_xf)_{L^2}.
\]
More in general, in any space dimension, we have that some summands in 
\[
-\sum_{k=1}^n(\partial_{x_k}(\wt{Q}\wt{A_k})V,V)_{L^2}
\]
will be cancelled by the terms in $((\wt{Q}\wt{B}+\wt{B}^\ast\wt{Q})V,V)_{L^2}$ where the derivatives $\partial_{x_k}a_k$ appear, for $k=1,\dots,n$. Indeed,
\begin{multline*}
-\sum_{k=1}^n(\partial_{x_k}(\wt{Q}\wt{A_k})V,V)_{L^2}=-2\sum_{k=1}^n\sum_{j=0}^{n-1} (\partial_{x_k}a_kV_{k+j(n+1)}, V_{(j+1)(n+1)})_{L^2}\\
=-2\sum_{0\le j\neq k-1\le n}(\partial_{x_k}a_kV_{k+j(n+1)}, V_{(j+1)(n+1)})_{L^2}-2\sum_{k=1}^n(\partial_{x_k}a_kV_{k+(k-1)(n+1)}, V_{k(n+1)})_{L^2}
\end{multline*}
and
\begin{multline*}
((\wt{Q}\wt{B}+\wt{B}^\ast\wt{Q})V,V)_{L^2}=2\sum_{1\le j\neq k\le n}(\partial_{x_k}a_jV_{j+(j-1)(n+1)},V_{k(n+1)})_{L^2}\\
+2\sum_{k=1}^n(\partial_{x_k}a_kV_{k+(k-1)(n+1)},V_{k(n+1)})_{L^2}.
\end{multline*}

\end{remark}
We can now estimates the terms in Proposition \ref{prop_terms_en_V}  above by means of the energy $E(t)$ defined by $\wt{Q}$ in \eqref{Energy_V}.
\begin{proposition}
\label{estimates_final}
Under the hypotheses (H1) and (H2),
\begin{itemize}
\item[(i)] there exists a constant $c_1(M,n)>0$ such that 
\[
\sum_{k=1}^n(\partial_{x_k}(\wt{Q}\wt{A_k})V,V)_{L^2}\le c_1E(t),
\]
for all $t\in[0,T]$;
\item[(ii)] there exists a constant $c_2(M,n)>0$ such that
\[
((\wt{Q}\wt{B}+\wt{B}^\ast\wt{Q})V,V)_{L^2}\le c_2 E(t).
\]
\end{itemize}
\end{proposition}
\begin{proof}
We begin by observing that the energy $E(t)$ can be rewritten as
\[
E(t)=\sum_{j=0}^{n-1}\sum_{k=1}^n(a_kV_{j(n+1)+k}, V_{j(n+1)+k})_{L^2}+\Vert V_{(j+1)(n+1)}\Vert_{L^2}^2.
\]
\begin{itemize}
\item[(i)]
By Proposition \ref{prop_terms_en_V} we have that
\[
\sum_{k=1}^n (\partial_{x_k}(\wt{Q}\wt{A_k})V,V)_{L^2}=2\sum_{k=1}^n\sum_{j=0}^n (\partial_{x_k}a_kV_{k+j(n+1)}, V_{(j+1)(n+1)})_{L^2}.
\]
By applying Cauchy-Schwarz and Glaeser's inequality we immediately obtain that
\begin{multline*}
2\sum_{j=0}^{n-1}\sum_{k=1}^n (\partial_{x_k}a_kV_{k+j(n+1)}, V_{(j+1)(n+1)})_{L^2}\\
\le \sum_{j=0}^{n-1}\sum_{k=1}^n 2M(a_kV_{k+j(n+1)},V_{k+j(n+1)})_{L^2}+\Vert V_{(j+1)(n+1)}\Vert_{L^2}^2\le c_1E(t).
\end{multline*}
\item[(ii)] Note that 
\[
\begin{split}
((\wt{Q}\wt{B}+\wt{B}^\ast\wt{Q})V,V)_{L^2}&=2\sum_{k=1}^n\sum_{j=1}^n(\partial_{x_k}a_jV_{j+(j-1)(n+1)}, V_{k(n+1)})_{L^2}\\
&=2\sum_{j=1}^n\sum_{k=1}^n(\partial_{x_j}a_kV_{k+(k-1)(n+1)}, V_{j(n+1)})_{L^2}
\end{split}
\]
and by Cauchy-Schwarz and Glaeser's inequality 
\begin{multline*}
2\sum_{j=1}^n\sum_{k=1}^n(\partial_{x_j}a_kV_{k+(k-1)(n+1)}, V_{j(n+1)})_{L^2}\\
\le 2M\sum_{k=1}^n(a_kV_{k+(k-1)(n+1)}, V_{k+(k-1)(n+1)})_{L^2}+\sum_{j=1}^n\Vert V_{j(n+1)}\Vert^2_{L^2}.
\end{multline*}
Since
\[
E(t)=\sum_{j=0}^{n-1}\sum_{k=1}^n(a_kV_{j(n+1)+k}, V_{j(n+1)+k})_{L^2}+\Vert V_{(j+1)(n+1)}\Vert_{L^2}^2.
\]
we easily see that 
\[
\sum_{k=1}^n(a_kV_{k+(k-1)(n+1)}, V_{k+(k-1)(n+1)})_{L^2}\le E(t)
\]
and
\[
\sum_{j=1}^n\Vert V_{j(n+1)}\Vert^2_{L^2}=\sum_{j=0}^{n-1}\Vert V_{(j+1)(n+1)}\Vert^2_{L^2}\le E(t).
\]
Thus, there exists a constant $c_2=c_2(M,n)>0$ such that 
\[
((\wt{Q}\wt{B}+\wt{B}^\ast\wt{Q})V,V)_{L^2}=2\sum_{k=1}^n\sum_{j=1}^n(\partial_{x_k}a_jV_{j+(j-1)(n+1)}, V_{k(n+1)})_{L^2}\le c_2E(t).
\]
\end{itemize}
\end{proof}
By combining Proposition \ref{prop_terms_en_V} with Proposition \ref{estimates_final} we obtain the following estimates on $E(t)$ and the components $V_{(j+1)(n+1)}$ of $V$ with $j=0,\dots, n-1$.
\begin{proposition}
\label{prop_LC_V}
Let 
\[
E(t)=(\wt{Q}V,V)_{L^2}=\sum_{j=0}^{n-1}\sum_{k=1}^n(a_kV_{j(n+1)+k}, V_{j(n+1)+k})_{L^2}+\Vert V_{(j+1)(n+1)}\Vert_{L^2}^2
\]
be the Energy of the system \eqref{CP1_V}. Under the assumptions (H1) and (H2) 
\begin{itemize}
\item[(i)] there exists a constant $c'=c'(n,M)>0$ such that 
\[
\frac{dE(t)}{dt}\le c'E(t)+\Vert f(t)\Vert_{H^1}^2,
\]
for all $t\in[0,T]$;
\item[(ii)] there exists a constant $C'=C'(n,M,\max_{i=1,\dots,n}\Vert a_i\Vert_{\infty},T)>0$ such that 
\[
\sum_{j=0}^{n-1}\Vert V_{(j+1)(n+1)}\Vert_{L^2}^2\le C'\biggl(\sum_{i=1}^{n(n+1)}\Vert V_i(0)\Vert_{L^2}^2+\int_{0}^t\Vert f(s)\Vert_{H^1}^2\, ds\biggr),
\]
for all $t\in[0,T]$.
\end{itemize}
\end{proposition}
\begin{proof}
By applying the results of Propositions \ref{prop_terms_en_V} and \ref{estimates_final} to \eqref{energy_V_dt} we have 
\[
\begin{split}
\frac{dE(t)}{dt}&\le c_1E(t)+c_2E(t)+2\sum_{k=1}^n (V_{k(n+1)}, \partial_{x_k}f)_{L^2}\\
&\le c_1E(t)+c_2E(t)+\Vert f\Vert_{H^1}^2+\sum_{k=1}^n \Vert V_{k(n+1)}\Vert_{L^2}^2\\
&\le (c_1+c_2+1)E(t)+\Vert f\Vert_{H^1}^2.
\end{split}
\]
Hence, by setting $c'=c_1+c_2+1$ we get the first assertion of this proposition. A straightforward application of  Gr\"onwall's lemma to the inequality above combined with the estimate \eqref{Energy_V_est_2} yields
\beq
\label{est_V_comp}
\begin{split}
\sum_{j=0}^{n-1}\Vert V_{(j+1)(n+1)}\Vert_{L^2}^2\le E(t)&\le \biggl(E(0)+\int_{0}^t\Vert f(s)\Vert_{H^1}^2\, ds\biggr){\rm e}^{c't}\\
&\le C'\biggl(\sum_{i=1}^{n(n+1)}\Vert V_i(0)\Vert_{L^2}^2+\int_{0}^t\Vert f(s)\Vert_{H^1}^2\, ds\biggr),
\end{split}
\eeq
for all $t\in[0,T]$, where $C'=C'(n,M,\max_{i=1,\dots,n}\Vert a_i\Vert_{\infty},T)>0$.
\end{proof}
We can now estimate the entries $U_i$ with $i\neq n+1$. Indeed, by construction, $\partial_t U_i=V_{(n+1)i}$ for $i=1,\dots,n$.
\begin{proposition}
\label{prop_est_U}
Let $i=1,\dots,n$. Under the assumptions (H1) and (H2) there exists a constant $C_i=C_i(n,M,N,\max_{i=1,\dots,n}\Vert a_i\Vert_{\infty},T)>0$ such that 
\beq
\label{est_U_i}
\Vert U_i(t)\Vert_{L^2}^2\le 
C_i\biggl(\Vert g_0\Vert_{H^2}^2+\Vert g_1\Vert_{H^1}^2+\int_{0}^t\Vert f(s)\Vert_{H^1}^2\, ds\biggr)\
\eeq
for all $t\in[0,T]$.
\end{proposition}
\begin{proof}
Let us write $\partial_tU_i$ as $V_{(n+1)i}$. By the fundamental theorem of calculus we have
\[
  \Vert U_i(t)\Vert^2_{L^2}\le 2\Vert U_i(t)-U_i(0)\Vert_{L^2}^2+2\Vert U_i(0)\Vert_{L^2}^2
 =2\Big\Vert \int_{0}^tV_{(n+1)i}\, ds\Big\Vert_{L^2}^2+2\Vert U_i(0)\Vert_{L^2}^2\\
 \]
  By Minkowski's integral inequality
  \[
  \Big\Vert \int_{0}^tV_{(n+1)i}\, ds\Big\Vert_{L^2}\le \int_{0}^t \Vert V_{(n+1)i}(s)\Vert_{L^2}ds
  \]
  and therefore
  \[
    \Vert U_i(t)\Vert^2_{L^2}\le 2\biggl(\int_{0}^t \Vert V_{(n+1)i}(s)\Vert_{L^2}ds\biggr)^2+2\Vert U_i(0)\Vert_{L^2}^2\le 2t^2 \sup_{s\in[0,t]} \Vert V_{(n+1)i}(s)\Vert^2_{L^2}+2\Vert U_i(0)\Vert_{L^2}^2.
  \]
  By making use of the estimate \eqref{est_V_comp} we conclude that 
  \[
    \Vert U_i(t)\Vert^2_{L^2}\le 2t^2C'\biggl(\sum_{i=1}^{n(n+1)}\Vert V_i(0)\Vert_{L^2}^2+\int_{0}^t\Vert f(s)\Vert_{H^1}^2\, ds\biggr)+2\Vert U_i(0)\Vert_{L^2}^2,
  \]
  where $C'=C'(n,M,\max_{i=1,\dots,n}\Vert a_i\Vert_{\infty},T)>0$, for all $t\in[0,T]$. Recalling that 
  \[
 V(0,x)=(\partial_{x_1}U(0,x),\partial_{x_2}U(0,x),\dots,\partial_{x_n}U(0,x), U(0,x))^T
 \]
 and 
 \[
 U(0,x)=(\partial_{x_1}g_0,\partial_{x_2}g_0,\cdots, \partial_{x_n}g_0, g_1)^{T}
 \]
 we easily see that
 \[
 \sum_{i=1}^{n(n+1)}\Vert V_i(0)\Vert_{L^2}^2\le \Vert g_0\Vert_{H^2}^2+\Vert g_1\Vert_{H^1}^2
 \]
 and
 \[
\Vert U_i(0)\Vert_{L^2}^2\le \Vert g_0\Vert^2_{H^1}.
 \]
 Hence, we have proven that there exists a new constant $C_i>0$ such that \eqref{est_U_i} holds.
  
\end{proof}
Summarising, we have proven so far that, under the assumptions (H1) and (H2),  if $U$ is a solution of the Cauchy problem 
\[
\begin{split}
\partial_tU&=\sum_{k=1}^nA_k(x)\partial_{x_k} U+F,\\
U(0,x)&=(\partial_{x_1}g_0,\partial_{x_2}g_0,\cdots, \partial_{x_n}g_0, g_1)^{T},
\end{split}
\]
then there exist constants $C_i=C_i(n,M,\max_{i=1,\dots,n}\Vert a_i\Vert_{\infty},T)>0$, for $i=1,\dots,n$ and $C_{n+1}=C_{n+1}(n,M,\max_{i=1,\dots,n}\Vert a_i\Vert_{\infty},T)>0$ such that 
\[
\begin{split}
\Vert U_i(t)\Vert_{L^2}^2&\le 
C_i\biggl(\Vert g_0\Vert_{H^2}^2+\Vert g_1\Vert_{H^1}^2+\int_{0}^t\Vert f(s)\Vert_{H^1}^2\, ds\biggr),\\
\Vert U_{n+1}(t)\Vert_{L^2}^2&\le C_{n+1}\biggl(\Vert g_0\Vert_{H^1}^2+\Vert g_1\Vert_{L^2}^2+\int_{0}^t\Vert f(s)\Vert_{L^2}^2\, ds\biggr),
\end{split}
\]
for all $t\in[0,T]$. It follows that there exists a constant $C_0(n,M,\max_{i=1,\dots,n}\Vert a_i\Vert_{\infty},T)>0$ such that
\[
\Vert U(t)\Vert_{L^2}^2\le 
C_0\biggl(\Vert g_0\Vert_{H^2}^2+\Vert g_1\Vert_{H^1}^2+\int_{0}^t\Vert f(s)\Vert_{H^1}^2\, ds\biggr).
\]
We now want to show that similar estimates can be obtained for all Sobolev norms.
 
\subsection{Sobolev estimates of $U$}
We begin by recalling that, in order to get $L^2$-estimates on $U$, we have transformed the $n+1\times n+1$ system 
\[
\begin{split}
\partial_tU&=\sum_{k=1}^nA_k(x)\partial_{x_k} U+F,\\
U(0,x)&=(\partial_{x_1}g_0,\partial_{x_2}g_0,\cdots, \partial_{x_n}g_0, g_1)^{T},
\end{split}
\]
into the $n(n+1)\times n(n+1)$  system
\beq
\label{system_V}
\begin{split}
\partial_t V&=\sum_{k=1}^n\wt{A_k}(x)\partial_{x_k} V+\wt{B}V+\wt{F},\\
V(0,x)&=(\partial_{x_1}U(0,x),\partial_{x_2}U(0,x),\dots,\partial_{x_n}U(0,x))^T,
\end{split}
\eeq
where, $V=\nabla_x U$, 
\[
\wt{A_k}=\left(
	\begin{array}{cccc}
	A_k & 0 & \cdots & 0\\
	0 & A_k & \cdots & 0\\
	\vdots & \vdots & \vdots & \vdots\\
	0 & 0 & \cdots & A_k
	\end{array}
	\right),
\]
for $k=1,\dots,n$,
\[
\wt{B}=\left(
	\begin{array}{ccccc}
	\partial_{x_1}A_1 & \partial_{x_1}A_2 & \cdots & \cdots & \partial_{x_1}A_n\\
	\partial_{x_2}A_1& \partial_{x_2}A_2 & \cdots & \cdots & \partial_{x_2}A_n\\
	\vdots & \vdots & \vdots & \vdots & \vdots\\
	\partial_{x_k}A_1&  \cdots &\partial_{x_k}A_k & \cdots & \partial_{x_k}A_n\\
	\vdots & \vdots & \vdots & \vdots & \vdots\\
	\partial_{x_n}A_1& \partial_{x_n}A_2 & \cdots & \cdots & \partial_{x_n}A_n\\
	 \end{array}
	\right)
\]
and $\wt{F}=\nabla_x F$. We now iterate this transformation. Each iteration will allow us to get Sobolev estimates on $U$ of one order higher.

Let $W=\nabla_x V$. If $V$ solves \eqref{system_V} then we obtain the $n^2(n+1)\times n^2(n+1)$ system
\beq
\label{system_W}
\begin{split}
\partial_t W&=\sum_{k=1}^n\wt{\wt{A_k}}(x)\partial_{x_k} W+\wt{\wt{B}}W+\wt{\wt{F}},\\
W(0,x)&=(\partial_{x_1}V(0,x),\partial_{x_2}V(0,x),\dots,\partial_{x_n}V(0,x))^T,
\end{split}
\eeq
where
\[
\wt{\wt{A_k}}=\left(
	\begin{array}{cccc}
	\wt{A_k} & 0 & \cdots & 0\\
	0 & \wt{A_k} & \cdots & 0\\
	\vdots & \vdots & \vdots & \vdots\\
	0 & 0 & \cdots & \wt{A_k}
	\end{array}
	\right),
\]
for $k=1,\dots,n$,
\[
\wt{\wt{B}}=\left(
	\begin{array}{ccccc}
	\partial_{x_1}\wt{A_1}+\wt{B} & \partial_{x_1}\wt{A_2} & \cdots & \cdots & \partial_{x_1}\wt{A_n}\\
	\partial_{x_2}\wt{A_1}& \partial_{x_2}\wt{A_2}+\wt{B} & \cdots & \cdots & \partial_{x_2}\wt{A_n}\\
	\vdots & \vdots & \vdots & \vdots & \vdots\\
	\partial_{x_k}\wt{A_1}&  \cdots &\partial_{x_k}\wt{A_k}+\wt{B} & \cdots & \partial_{x_k}\wt{A_n}\\
	\vdots & \vdots & \vdots & \vdots & \vdots\\
	\partial_{x_n}\wt{A_1}& \partial_{x_n}\wt{A_2} & \cdots & \cdots & \partial_{x_n}\wt{A_n}+\wt{B}\\
	 
	\end{array}
	\right)
\]
and 
\[
\wt{\wt{F}}=\nabla_x \wt{F}+\left(
	\begin{array}{c}
	(\partial_{x_1}\wt{B})V\\ 
	(\partial_{x_2}\wt{B})V\\ 
	\vdots \\
	(\partial_{x_k}\wt{B})V\\
	\vdots\\
	(\partial_{x_n}\wt{B})V\\ 
	\end{array}
	\right).
\]
This system has a structure similar to the one of \eqref{system_V}, so we define the corresponding energy $E(t)=(\wt{\wt{Q}}W,W)_{L^2}$, where $\wt{\wt{Q}}$ is a block-diagonal matrix with $n^2$ identical blocks equal to $Q$. Arguing as in \eqref{energy_V_dt} we get 
\[
 \frac{dE(t)}{dt} =-\sum_{k=1}^n(\partial_{x_k}(\wt{\wt{Q}}\wt{\wt{A_k}})W,W)_{L^2}+((\wt{\wt{Q}}\wt{\wt{B}}+\wt{\wt{B}}^\ast\wt{\wt{Q}})W,W)_{L^2}+2(\wt{\wt{Q}}W,\wt{\wt{F}})_{L^2}.
\]
\begin{proposition}
\label{estimates_final_W}
Under the hypotheses (H1) and (H2),
\begin{itemize}
\item[(i)] there exists a constant $c_1(M,n)>0$ such that 
\[
\sum_{k=1}^n(\partial_{x_k}(\wt{\wt{Q}}\wt{\wt{A_k}})W,W)_{L^2}\le c_1E(t),
\]
for all $t\in[0,T]$;
\item[(ii)] there exists a constant $c_2(M,n)>0$ such that
\[
((\wt{\wt{Q}}\wt{\wt{B}}+\wt{\wt{B}}^\ast\wt{\wt{Q}})W,W)_{L^2}\le c_2 E(t).
\]
\item[(iii)] there exist 
\[
c_3(n, \max_{i=1,\dots,n, |\alpha|=2}\Vert \partial^\alpha a_i\Vert_{\infty}^2,T)>0
\]
such that 
\[
2(\wt{\wt{Q}}W,\wt{\wt{F}})_{L^2}\le 2E(t)+\Vert f\Vert_{H^2}^2+c_3\biggl(\int_0^t E(s)\, ds+\Vert g_0\Vert_{H^2}^2\biggr).
\]
\end{itemize}
\end{proposition}
\begin{proof}
\leavevmode
\begin{itemize}
\item[(i)] Since the matrices $\wt{\wt{Q}}$ and $\wt{\wt{A_k}}$ are block diagonal we can argue on any block as for $(\partial_{x_k}({\wt{Q}}{\wt{A_k}})V,V)_{L^2}$ and we obtain the desired estimate.
\item[(ii)] Note that $((\wt{\wt{Q}}\wt{\wt{B}}+\wt{\wt{B}}^\ast\wt{\wt{Q}})W,W)_{L^2}=2(\wt{\wt{Q}}\wt{\wt{B}}W,W)_{L^2}$. We have that
\begin{multline*}
\wt{\wt{Q}}\wt{\wt{B}}W=\\
\left(
	\begin{array}{ccccc}
	\wt{Q}\partial_{x_1}\wt{A_1} & \wt{Q}\partial_{x_1}\wt{A_2} & \cdots & \cdots & \wt{Q}\partial_{x_1}\wt{A_n}\\
	\wt{Q}\partial_{x_2}\wt{A_1}& \wt{Q}\partial_{x_2}\wt{A_2}& \cdots & \cdots &\wt{Q}\partial_{x_2}\wt{A_n}\\
	\vdots & \vdots & \vdots & \vdots & \vdots\\
	\wt{Q}\partial_{x_k}\wt{A_1}&  \cdots &\wt{Q}\partial_{x_k}\wt{A_k} & \cdots & \wt{Q}\partial_{x_k}\wt{A_n}\\
	\vdots & \vdots & \vdots & \vdots & \vdots\\
	\wt{Q}\partial_{x_n}\wt{A_1}& \wt{Q}\partial_{x_n}\wt{A_2} & \cdots & \cdots & \wt{Q}\partial_{x_n}\wt{A_n}\\
	 \end{array}
	\right)W+\left(
	\begin{array}{cccc}
	\wt{Q}\wt{B} & 0 & \cdots & 0\\
	0 & \wt{Q}\wt{B}  & \cdots & 0\\
	\vdots & \vdots & \vdots & \vdots\\
	0 & 0 & \cdots & \wt{Q}\wt{B} 
	\end{array}
	\right)W\\
	=S_1W+S_2W
\end{multline*}
Since the second summand $S_2$ is block diagonal and our energy is block diagonal as well we can argue on it as for $(\wt{Q}\wt{B}V,V)$ obtaining immediately the estimate we want, i.e., $(S_2W,W)_{L^2}\le c_2 E(t)$. By direct computations on the first summand $S_1$ we have that 
\[
(S_1W,W)_{L^2}=\sum_{k=1}^n\sum_{j=1}^n\sum_{h=1}^n(\partial_{x_j}a_h W_{h+n(n+1)(h-1)+(k-1)(n+1)}, W_{(n+1)k+n(n+1)(j-1)}).
\]
Since by Cauchy-Schwarz and Glaeser's inequality 
\[
\begin{split}
&2(\partial_{x_j}a_h W_{h+n(n+1)(h-1)+(k-1)(n+1)}, W_{(n+1)k+n(n+1)(j-1)})\\
&\le 2M(a_h W_{h+n(n+1)(h-1)+(k-1)(n+1)}, W_{h+n(n+1)(h-1)+(k-1)(n+1)})_{L^2}\\
&+\Vert W_{(n+1)k+n(n+1)(j-1)})\Vert_{L^2}^2
\end{split}
\]
we deduce that there exists a constant $c_2(M,n)$ such that 
\[
\begin{split}
&2(S_1W,W)\\
&\le \sum_{k=1}^n\sum_{j=1}^n\sum_{h=1}^n \biggl(2M(a_h W_{h+n(n+1)(h-1)+(k-1)(n+1)}, W_{h+n(n+1)(h-1)+(k-1)(n+1)})_{L^2}\\
&+\Vert W_{(n+1)k+n(n+1)(j-1)})\Vert_{L^2}^2\biggr)\le c_2(\wt{\wt{Q}}W,W)_{L^2}=c_2 E(t),
\end{split}
\]
for all $t\in[0,T]$. This proves assertion (ii).
\item[(iii)] We begin by writing $(\wt{\wt{Q}}W,\wt{\wt{F}})_{L^2}$ as 
\[
(\wt{\wt{Q}}W,\nabla_x \wt{F})_{L^2} +(\wt{\wt{Q}}W, \left(
	\begin{array}{c}
	(\partial_{x_1}\wt{B})V\\ 
	(\partial_{x_2}\wt{B})V\\ 
	\vdots \\
	(\partial_{x_k}\wt{B})V\\
	\vdots\\
	(\partial_{x_n}\wt{B})V\\ 
	\end{array}
	\right))_{L^2}=(\wt{\wt{Q}}W, T_1)_{L^2}+(\wt{\wt{Q}}W, T_2)_{L^2}.
\]
By definition of $\wt{F}$ and $\nabla_x\wt{F}$ we easily see that 
\[
(\wt{\wt{Q}}W, T_1)_{L^2}=\sum_{k=1}^n\sum_{j=1}^n  (W_{(n+1)k+n(n+1)(j-1)}, \partial_{x_k}\partial_{x_j}f)_{L^2}.
\]
Since the coefficients $a_i$ are positive for $i=1, \dots, n$ we have that the energy $E(t)$ can be bounded from below as follows
\[
\sum_{k=1}^n\sum_{j=1}^n  \Vert W_{(n+1)k+n(n+1)(j-1)}\Vert_{L^2}^2\le E(t),
\]
thus
\beq
\label{est_T_1}
2(\wt{\wt{Q}}W, T_1)_{L^2}\le \sum_{k=1}^n\sum_{j=1}^n  \Vert W_{(n+1)k+n(n+1)(j-1)}\Vert_{L^2}^2 +\Vert f\Vert_{H^2}^2\le E(t)+\Vert f\Vert_{H^2}^2.
\eeq
We now want to estimate $2(\wt{\wt{Q}}W, T_2)_{L^2}$. Note that the matrices $\partial_{x_i}\wt{B}$ are defined by the second order derivatives of the coefficients and have only the rows with index $(j+1)(n+1)$, $j=0,\dots,n-1$, not identically zero. Hence, by direct computations and by arguing as in one-dimensional case we get that there exists a constant $c(n)>0$ such that 
\begin{multline}
\label{est_T_2}
2(\wt{\wt{Q}}W, T_2)_{L^2}\le  \sum_{k=1}^n\sum_{j=1}^n  \Vert W_{(n+1)k+n(n+1)(j-1)}\Vert_{L^2}^2\\
+c\max_{i=1,\dots,n, |\alpha|=2}\Vert \partial^\alpha a_i\Vert_{\infty}^2\sum_{k=1}^n\sum_{j=1}^n\Vert V_{k+(n+1)(j-1)}\Vert_{L^2}^2.
\end{multline}
Since by definition of $U$, $V$ and $W$,
\[
W_{(n+1)k+n(n+1)(j-1)}=\partial_t V_{k+(n+1)(j-1)},
\]
for $1\le k,j\le n$, by the fundamental theorem of calculus combined with Cauchy-Schwarz and the Minkowski's inequality in integral form, arguing as in \eqref{est_E_W_3} in the one-dimensional case, we can deduce from \eqref{est_T_2} the following inequality:
\begin{multline*}
2(\wt{\wt{Q}}W, T_2)_{L^2}\le E(t)+c\max_{i=1,\dots,n, |\alpha|=2}\Vert \partial^\alpha a_i\Vert_{\infty}^2\\
\sum_{k=1}^n\sum_{j=1}^n\biggl(2t^2\int_0^t\Vert W_{(n+1)k+n(n+1)(j-1)}(s)\Vert^2_{L^2}\, ds+2\Vert V_{k+(n+1)(j-1)}(0)\Vert_{L^2}^2\biggr)\\
\le E(t)+c(n,T)\max_{i=1,\dots,n, |\alpha|=2}\Vert \partial^\alpha a_i\Vert_{\infty}^2\biggl(\int_{0}^t E(s)\, ds+\sum_{k=1}^n\sum_{j=1}^n\Vert V_{k+(n+1)(j-1)}(0)\Vert_{L^2}^2\biggr)\\
\le E(t)+c(n,T)\max_{i=1,\dots,n, |\alpha|=2}\Vert \partial^\alpha a_i\Vert_{\infty}^2\biggl(\int_{0}^t E(s)\, ds+\Vert g_0\Vert_{H^2}^2\biggr).
\end{multline*}
Combining \eqref{est_T_1} with the estimate above we obtain that  
\begin{multline*}
2(\wt{\wt{Q}}W,\wt{\wt{F}})_{L^2}\le 2E(t)+\Vert f\Vert_{H^2}^2\\+ c_3(n,T, \max_{i=1,\dots,n, |\alpha|=2}\Vert \partial^\alpha a_i\Vert_{\infty}^2)\biggl(\int_0^t E(s)\, ds+\Vert g_0\Vert_{H^2}^2\biggr).
\end{multline*}
This proves assertion (iii).
\end{itemize}
\end{proof}
We can now apply the estimates of the previous proposition to 
\[
 \frac{dE(t)}{dt} =-\sum_{k=1}^n(\partial_{x_k}(\wt{\wt{Q}}\wt{\wt{A_k}})W,W)_{L^2}+((\wt{\wt{Q}}\wt{\wt{B}}+\wt{\wt{B}}^\ast\wt{\wt{Q}})W,W)_{L^2}+2(\wt{\wt{Q}}W,\wt{\wt{F}})_{L^2}.
\]
We get that there exists a constant $c'>0$ depending on $M,n,T$ and $\max_{i=1,\dots,n, |\alpha|=2}\Vert \partial^\alpha a_i\Vert_{\infty}^2$ such that 
\beq
\label{E_V_fin}
 \frac{dE(t)}{dt} \le c'\biggl(E(t)+\int_0^t E(s)\, ds+\Vert f(t)\Vert_{H^2}^2+\Vert g_0\Vert_{H^2}^2\biggl). 
\eeq
For the sake of the reader we now recall a Gr\"onwall's type lemma (Lemma 6.2 in \cite{ST}) that will be applied to the inequality \eqref{E_V_fin} in order to estimate the energy $E(t)$.  
\begin{lemma}
\label{lem_ST_G}
Let $\varphi\in C^1([0,T])$ and $\psi\in C([0,T])$ two positive functions such that 
\[
\varphi'(t)\le C_1\varphi(t)+C_2\int_0^t\varphi(s)\, ds+\psi(t),\qquad t\in[0,T],
\]
for some constants $C_1,C_2>0$. Then, there exists a constant $C=C(C_1,C_2,T)>0$ such that 
\[
\varphi(t)\le C\biggl(\varphi(0)+\int_0^t\psi(s)\, ds\biggl),
\]
for all $t\in[0,T]$.
\end{lemma}
Note that the constant $C$ depends exponentially on $C_1, C_2$ and $T$, so there exists a constant $C'>0$ depending exponentially on $M,n,T$ and $\max_{i=1,\dots,n, |\alpha|=2}\Vert \partial^\alpha a_i\Vert_{\infty}^2$ such that 
\beq
\label{est_E_W}
E(t)\le C'\biggl(E(0)+\int_0^t \Vert f(s)\Vert_{H^2}^2\, ds+\Vert g_0\Vert_{H^2}^2\biggr).
\eeq
We recall that by construction of the energy $E(t)=(\wt{\wt{Q}}W,W)$ we have the bound from below
\beq
\label{bfb_W}
\sum_{k=1}^n\sum_{j=1}^n\Vert W_{(n+1)k+n(n+1)(j-1)})\Vert_{L^2}^2\le E(t).
\eeq
Combining \eqref{bfb_W} with \eqref{est_E_W} we can write
\[
\sum_{k=1}^n\sum_{j=1}^n\Vert W_{(n+1)k+n(n+1)(j-1)}(t)\Vert_{L^2}^2\le C'\biggl(E(0)+\int_0^t \Vert f(s)\Vert_{H^2}^2\, ds+\Vert g_0\Vert_{H^2}^2 \biggr).
\]
Hence, there exists a constant $C''(n,M,\max_{i=1,\dots,n}\Vert a_i\Vert_{\infty}, \max_{i=1,\dots,n, |\alpha|=2}\Vert \partial^\alpha a_i\Vert_{\infty}^2,T)>0$ such that 
\beq
\label{est_W_1}
\begin{split}
&\sum_{k=1}^n\sum_{j=1}^n\Vert W_{(n+1)k+n(n+1)(j-1)}(t)\Vert_{L^2}^2\le C'' \biggl(\sum_{k=1}^{n^2(n+1)}\Vert W_k(0)\Vert_{L^2}^2+\int_0^t \Vert f(s)\Vert_{H^2}^2\, ds+\Vert g_0\Vert_{H^2}^2\biggr)\\
&\le C''\biggl(\Vert g_0\Vert_{H^3}^2+\Vert g_1\Vert_{H^2}^2+\int_0^t \Vert f(s)\Vert_{H^2}^2\, ds).
\end{split}
\eeq
Note that in the estimate above we have used the definition of the initial data $W(0)$. 
Since by definition of $U$, $V$ and $W$,
\[
W_{(n+1)k+n(n+1)(j-1)}=\partial_t V_{k+(n+1)(j-1)},
\]
for $1\le k,j\le n$, arguing as in the proof of Proposition \ref{prop_est_U} we can transfer the estimate \eqref{est_W_1} from $W$ to $V$. More precisely,for a suitable constant 
\[
C''=C''(n,M,\max_{i=1,\dots,n}\Vert a_i\Vert_{\infty}, \max_{i=1,\dots,n, |\alpha|=2}\Vert \partial^\alpha a_i\Vert_{\infty}^2,T)>0,
\]
we have that
\beq
\label{est_V_complete}
\sum_{k=1}^n\sum_{j=1}^n\Vert V_{k+(n+1)(j-1)}(t)\Vert_{L^2}^2\le C''\biggl(\Vert g_0\Vert_{H^3}^2+\Vert g_1\Vert_{H^2}^2+\int_0^t \Vert f(s)\Vert_{H^2}^2\, ds\biggl),
\eeq
for all $t\in[0,T]$. Combining \eqref{est_V_comp} with \eqref{est_V_complete} we have found an estimate for $\Vert V\Vert^2_{L^2}$ and recalling that $V=\nabla U$ we conclude that there exists a constant
\[
C_1=C_1(n,M,\max_{i=1,\dots,n}\Vert a_i\Vert_{\infty}, \max_{i=1,\dots,n, |\alpha|=2}\Vert \partial^\alpha a_i\Vert_{\infty}^2,T)>0
\]
such that 
\beq
\label{est_U_H_1}
\Vert U(t)\Vert_{H^1}^2\le C_1\biggl(\Vert g_0\Vert_{H^3}^2+\Vert g_1\Vert_{H^2}^2+\int_0^t \Vert f(s)\Vert_{H^2}^2\, ds\biggl),
\eeq
 for all $t\in[0,T]$. The constant $C_1$ depends linearly on $\max_{i=1,\dots,n}\Vert a_i\Vert_{\infty}$ and exponentially on all the rest. As in the one-dimensional case (see \cite{ST}), this argument can be easily iterated taking an extra derivative with respect to $x$ at every step. One gets systems with the same structure where the right-hand side depends on higher order derivatives of the coefficients $a_j$. More precisely, if we want to estimate the $H^k$-norm of $U(t)$ then we will derive the coefficients $a_j$ up to order $k+1$. We therefore have the following proposition.
 \begin{proposition}
 \label{prop_est_Sob_k}
 Assume that the coefficients $a_i\ge 0$ are smooth and bounded with bounded derivatives of any order. Then, for all $k\in\N$ there exists a constant $C_k$ depending on $T$, $M$ and the $L^\infty$-norms of the derivatives of the coefficients up to order $k+1$ such that
 \beq
\label{est_U_H_k}
\Vert U(t)\Vert_{H^k}^2\le C_k\biggl(\Vert g_0\Vert_{H^{k+2}}^2+\Vert g_1\Vert_{H^{k+1}}^2+\int_0^t \Vert f(s)\Vert_{H^{k+1}}^2\, ds\biggl),
\eeq
for all $t\in[0,T]$.
 \end{proposition}
 Note that this is an extension to any space dimension of the estimate (1.32) in \cite{ST} in the case $m=2$ for global well-posedness.

\subsection{Existence and uniqueness result}
We conclude this section by proving that the Cauchy problem \eqref{CP_wave}
\[
\begin{split}
\partial_t^2u-\sum_{i=1}^na_i(x)\partial^2_{x_i} u&=f(t,x),\quad t\in[0,T], x\in\R^n,\\
u(0,x)&=g_0,\\
\partial_tu(0,x)&=g_1,
\end{split}
\]
is well-posed in every Sobolev space and therefore in $C^\infty$. It will be crucial to employ the estimate \eqref{est_U_H_k} which can be re-written in terms of $u$ as
\beq
\label{est_u_H_k}
\Vert u(t)\Vert_{H^{k+1}}^2\le C_k\biggl(\Vert g_0\Vert_{H^{k+2}}^2+\Vert g_1\Vert_{H^{k+1}}^2+\int_0^t \Vert f(s)\Vert_{H^{k+1}}^2\, ds\biggl),
\eeq
\begin{theorem}
\label{theo_main}
Let the coefficients $a_i\ge 0$ be smooth and bounded with bounded derivatives of any order. Assume that $f\in C([0,T], H^{\infty})$. Then, the Cauchy problem \eqref{CP_wave} is well-posed in every Sobolev space $H^k$, with $k\in\N$, and 
or all $k\in\N$ there exists a constant $C_k$ depending on $T$, $M$ and the $L^\infty$-norms of the derivatives of the coefficients up to order $k+1$ such that
 \beq
\label{est_u_H_k_1}
\Vert u(t)\Vert_{H^{k+1}}^2\le C_k\biggl(\Vert g_0\Vert_{H^{k+2}}^2+\Vert g_1\Vert_{H^{k+1}}^2+\int_0^t \Vert f(s)\Vert_{H^{k+1}}^2\, ds\biggl),
\eeq
for all $t\in[0,T]$.
\end{theorem}
\begin{proof}

\leavevmode
\begin{itemize}
 \item[(i)] {\bf Existence}. Let 
 \[
 P(u)=\partial^2_tu-\sum_{i=1}^na_i(x)\partial^2_{x_i} u.
 \]
 Assume that $f\in C([0,T], H^1(\R^n))$. The strictly hyperbolic Cauchy problem 
 \[
 \begin{split}
 P_\delta(u)&=\partial^2_tu-\sum_{i=1}^n(a_i(x)+\delta)\partial^2_{x_i} u=f\\
 u(0,x)&=g_0(x)\in H^2(\R^n),\\
\partial_t u(0,x)&=g_1(x)\in H^1(\R^n),
\end{split}
 \]
 has a unique solution $(u_\delta)_\delta$ defined via the corresponding vector $(U_\delta)_\delta$. Since, the constant $C=C(t, n, \Vert a_1+\delta\Vert_{L^\infty}, \Vert a_2+\delta\Vert_{L^\infty}, M)>0$ can be chosen independent of $\delta\in(0,1)$ we have that, given $g_0(x)\in H^2(\R^n)$ and $g_1(x)\in H^1(\R^n)$, the net 
 \[
 U_\delta=(U_{1,\delta}, U_{2,\delta}, \cdots, U_{n,\delta}) 
 \]
 is bounded in $(L^2(\R^n))^{n+1}$. Therefore there exists a convergent subsequent in $(L^2(\R^n))^{n+1}$ with limit $U\in (L^2(\R^n))^{n+1}$ solving the system 
\[
\begin{split}
\partial_tU&=\sum_{k=1}^nA_k(x)\partial_{x_k} U+F,\\
U(0,x)&=(\partial_{x_1}g_0,\partial_{x_2}g_0,\cdots, \partial_{x_n}g_0, g_1)^{T},
\end{split}
\]
in the sense of distributions. $U$ gives a solution $u\in C^2([0,T], H^1(\R^n))$ of the original Cauchy problem \eqref{CP_wave}. The same argument can be iterated (as for the systems in $V$ and $W$) to show that if $g_0\in H^{k+2}(\R^n)$, $g_1\in H^{k+1}(\R^n)$ and $f\in C([0,1], H^{k+1}(\R^n))$ then \eqref{CP_wave} has a solution  $u\in C^2([0,T], H^{k+1}(\R^n))$.

 \item[(ii)] {\bf Uniqueness}. The uniqueness of the solution $u$ follows immediately from the estimate \eqref{est_u_H_k_1}. 
 \end{itemize}

\end{proof}
The following result of $C^\infty$ well-posedness is a straightforward corollary. It is consistent with Oleinik's result \cite{O70} and it is an extension to any space dimension of the one well-posedness result obtained in \cite{ST}. Note that since we take here bounded coefficients we have global well-posedness rather than local well-posedness.
\begin{corollary}
\label{cor_C_infty}
Let the coefficients $a_i\ge 0$ be smooth and bounded with bounded derivatives of any order and let $f\in C([0,T], C^\infty_c(\R^n))$. Then the Cauchy problem \eqref{CP_wave} is $C^\infty$ well-posed, i.e., given $g_0,g_1\in C^\infty_c(\R^n)$ there exists a unique solution $C^2([0,1], C^\infty(\R^n))$ of 
\[
\begin{split}
\partial_t^2u-\sum_{i=1}^na_i(x)\partial^2_{x_i} u&=f(t,x),\quad t\in[0,T], x\in\R^n,\\
u(0,x)&=g_0,\\
\partial_tu(0,x)&=g_1.
\end{split}
\]
Moreover, the estimate \eqref{est_u_H_k_1} holds for every $k\in\N$.
\end{corollary}
Note that by finite speed of propagation it is possible to remove the assumption of compact support on $f$ and the initial data.

\section{Very weak well-posedness}
\label{sec_CP_vw}
The final part of the paper is devoted to the case of non-regular coefficients. More precisely, we assume that 
\begin{itemize}
\item $a_i\in\E'(\R^n)$ with $a_i\ge 0$ for all $i=1,\dots ,n$,
\item $f\in C^\infty([0,T],\E'(\R^n))$,
\item $g_0,g_1\in\E'(\R^n)$.
\end{itemize}
Making use of the well-posedness result obtained in the previous section and of the regularisation techniques introduced at the beginning of the paper we will prove that the Cauchy problem
\[
\begin{split}
\partial_t^2u-\sum_{i=1}^na_i(x)\partial^2_{x_i} u&=f(t,x),\quad t\in[0,T], x\in\R^n,\\
u(0,x)&=g_0,\\
\partial_tu(0,x)&=g_1.
\end{split}
\]
is very weakly well-posed. 

Let $\varphi$ be a mollifier, i.e., $\varphi\in\Cinfc(\R^n)$ with $\int \varphi(x)\, dx=1$. Assume in addition that $\varphi\ge 0$. Let $\omega(\eps)$ be a positive scale (see Remark \ref{rem_net}). In the sequel, we will use the brief notations
\[
\begin{split}
a_{i,\eps}(x)&=(a_i\ast\varphi_{\omega(\eps)})(x),\qquad i=1,\dots,n,\\
f_\eps(t,x)&=(f(t,\cdot)\ast\varphi_{\eps})(x),\\
g_{0,\eps}(x)&=(g_0\ast\varphi_{\eps})(x),\\
g_{1,\eps}(x)&=(g_1\ast\varphi_{\eps})(x).\\
\end{split}
\]
By combining Proposition \ref{prop_reg_nets} and Proposition \ref{prop_nets_asymp_2} we obtain the following moderateness result.
\begin{proposition}
\label{prop_mod_Cauchy}
Let $\varphi$ be a mollifier with $\varphi\ge 0$. Let $\omega(\eps)$ be a positive scale. Then,  for all $i=1,\dot,n$ 
\begin{itemize}
\item[(i)]  $a_{i,\eps}\ge 0$ and there exists $N\in\N$ and for all $\alpha\in\N^n$ there exists $c>0$ such that 
\[
|\partial^\alpha a_{i,\eps}(x)|\le c\,\omega(\eps)^{-N-|\alpha|},
\]
for all $x\in\R^n$ and $\eps\in(0,1]$;
\item[(ii)] there exists $N'\in\N$ and for all $k\in\N$ and $\alpha\in\N^n$ there exists $c'>0$ such that 
\[
|\partial^k_t\partial_x^\alpha f_\eps(t,x)|\le c'\,\omega(\eps)^{-N'-|\alpha|},
\]
for all $t\in[0,T]$, $x\in\R^n$ and $\eps\in(0,1]$;

\item[(iii)] for all $j=0,1$ there exists $N_j\in\N$ and for all $\alpha\in\N^n$ there exists $c_j>0$ such that 
\[
|\partial^\alpha g_j(x)|\le c_j\,\omega(\eps)^{-N_j-|\alpha|},
\]
for all $x\in\R^n$ and $\eps\in(0,1]$.
\end{itemize}
\end{proposition}
By applying Proposition \ref{prop_Glaeser} to $a_{i,\eps}$ we obtain the following Glaeser's inequality: there exists $N\in\N$ and a constant $c>0$ such that 
\beq
\label{Glaeser_eps}
|\partial_{x_j}a_{i,\eps}(x)|^2\le 2c\,\omega(\eps)^{-N-2}a_{i,\eps}(x)
\eeq
for all $i,j=1,\dots,n$, $x\in\R^n$ and $\eps\in(0,1]$. In other words, the constant $M>0$ appearing in the classical Glaeser's inequality is here depending on $\eps$ and equal to 
\beq
\label{M_eps}
c\,\omega(\eps)^{-N-2}.
\eeq

We can now apply Theorem \ref{theo_main} to the regularised Cauchy problem 
\beq
\label{CP_eps_last}
\begin{split}
\partial_t^2u-\sum_{i=1}^na_{i,\eps}(x)\partial^2_{x_i} u&=f_\eps(t,x),\quad t\in[0,T], x\in\R^n,\\
u(0,x)&=g_{0,\eps}\\
\partial_tu(0,x)&=g_{1,\eps}.
\end{split}
\eeq
and obtain immediately very weak well-posedness provided that we choose the net $\omega$ suitably. More precisely, we get the following result at the net-level.
\begin{proposition}
\label{prop_net_vww}
Under the assumptions above, the Cauchy problem \eqref{CP_eps_last} has a unique solution $(u_\eps)_\eps\in C^2([0,1], C^\infty(\R^n))^{(0,1]}$ such that 
for all $k\in\N$ there exists a net $C_{\eps,k}>0$ depending on $T$, $\omega(\eps)^{-N-2}$ as in \eqref{M_eps} and the $L^\infty$-norms of the derivatives of the coefficients $a_{i,\eps}$ up to order $k+1$ such that
 \beq
\label{est_u_H_k}
\Vert u_\eps(t)\Vert_{H^{k+1}}^2\le C_{k,\eps}\biggl(\Vert g_{0,\eps}\Vert_{H^{k+2}}^2+\Vert g_{1,\eps}\Vert_{H^{k+1}}^2+\int_0^t \Vert f_\eps(s)\Vert_{H^{k+1}}^2\, ds\biggl),
\eeq
for all $t\in[0,T]$ and $\eps\in(0,1]$.

\end{proposition}
Note that the net $C_{k,\eps}$ depends on $\omega(\eps)^{-N-2}$ in exponential form, i.e., ${\rm e}^{c\omega(\eps)^{-N-2}}$ for some constant $c>0$ (see for instance \eqref{est_U_n+1}). The same is true for the $L^\infty$-norms of the derivatives of the coefficients $a_{i,\eps}$ up to order $k+1$ (see \eqref{est_E_W}). In other words, we have that there exists $c_k>0$ and $N_k\in\N$ such that 
\[
C_{k,\eps}={\rm e}^{c_k\omega(\eps)^{-N_k}}.
\]
It follows that we need to choose $\omega(\eps)$ of logarithmic type to be sure that $(u_\eps)_\eps$ is moderate given moderate initial data and right-hand side. More precisely, let us choose $\omega(\eps)$ with 
\[
\omega^{-1}(\eps)=\ln(\ln(\eps^{-1})),
\]
for $\eps\in(0,1/2]$ and $\omega(\eps)=1$ for $\eps\in(1/2,1)$. By direct computations we have that for every $r>0$ there exists $c_r>0$ such that 
\[
\ln(\ln(\eps^{-1}))\le c_r\ln^r(\eps^{-1}).
\]
Hence,
\[
\omega(\eps)^{-N_k}\le (c_{N_k})^{N_k}(\ln^{\frac{1}{N_k}}(\eps^{-1}))^{N_k}=c'_{N_k}\ln(\eps^{-1}) 
\]
and
\[
C_{k,\eps}={\rm e}^{c_k\omega(\eps)^{-N_k}}\le{\rm e}^{c_kc'_{N_k}\ln(\eps^{-1}) }=\eps^{-c_kc'_{N_k}}.
\]
So, from \eqref{est_u_H_k} and Sobolev embedding inequalities, choosing $\omega(\eps)$ as above we have that if the initial data and the right-hand side are moderate then the solution net $(u_\eps)_\eps$ is moderate. Analogously, the same will be true replacing moderate with negligible. Note that one could choose different mollifiers for regularising coefficients and initial data. The argument above will still work provided that the mollifier we use in regularising $a$ is positive.  We can therefore state the following result of very weak well-posedness. 

\begin{theorem}
\label{theo_vw_well-posed}
The Cauchy problem
\[
\begin{split}
\partial_t^2u-\sum_{i=1}^na_i(x)\partial^2_{x_i} u&=f(t,x),\quad t\in[0,T], x\in\R^n,\\
u(0,x)&=g_0,\\
\partial_tu(0,x)&=g_1,
\end{split}
\]
where  $a_i\in\E'(\R^n)$ with $a_i\ge 0$ for all $i=1,\dots ,n$, $f\in C^\infty([0,T],\E'(\R^n))$ and $g_0,g_1\in\E'(\R^n)$, is very weakly well-posed, i.e., a very weak solution exists and it is unique modulo negligible nets (negligible changes in the regularisation of the equation coefficients and initial data lead to negligible changes in the corresponding very weak solution).
\end{theorem}
\begin{proof}
The existence proof is given above. Concerning uniqueness (in the very weak sense), let $(u_\eps)_\eps$ be the solution constructed above and let $(u'_\eps)_\eps$ be the solution of the Cauchy problem where coefficients and initial data have been perturbed by negligible nets, i.e., 
\[
\begin{split}
\partial_t^2u-\sum_{i=1}^na'_{i,\eps}(x)\partial^2_{x_i} u&=f'_\eps(t,x),\quad t\in[0,T], x\in\R^n,\\
u(0,x)&=g'_{0,\eps}\\
\partial_tu(0,x)&=g'_{1,\eps},
\end{split}
\]
where all the nets $(a_{i,\eps}-a'_{i,\eps})_\eps$, $(f_\eps-f'_\eps)_\eps$, $(g_{0,\eps}-g'_{0,\eps})_\eps$ and $(g_{1,\eps}-g'_{1,\eps})_\eps$ are negligible (in the appropriate function spaces). Hence, for $v_\eps=u_\eps-u'_\eps$ we have 
\[
\begin{split}
\partial_t^2v_\eps-\sum_{i=1}^na_{i,\eps}(x)\partial^2_{x_i} v_\eps&=\sum_{i=1}^n(a_{i,\eps}-a'_{i,\eps})(x)\partial_{x_i}^2u'_\eps+( f_\eps-f'_\eps)(t,x),\\
v_\eps(0,x)&=g_{0,\eps}-g'_{0,\eps}\\
\partial_tv_\eps(0,x)&=g_{1,\eps}-g'_{1,\eps},
\end{split}
\]
and therefore from \eqref{est_u_H_k}
\[
\Vert v_\eps(t)\Vert_{H^{k+1}}^2\le C_{k,\eps}\biggl(\Vert g_{0,\eps}-g'_{0,\eps}\Vert_{H^{k+2}}^2+\Vert g_{1,\eps}-g_{1,\eps}\Vert_{H^{k+1}}^2+\int_0^t \Vert F_\eps(s)\Vert_{H^{k+1}}^2\, ds\biggl),
\]
with
\[
F_\eps(s,x)=\sum_{i=1}^n(a_{i,\eps}-a'_{i,\eps})(x)\partial_{x_i}^2u'_\eps(s,x)+( f_\eps-f'_\eps)(s,x).
\]
Since $(h_\eps)_\eps$ is negligible as well as $(g_{0,\eps}-g'_{0,\eps})_\eps$ and $(g_{1,\eps}-g_{1,\eps})_\eps$ we conclude that $(v_\eps)_\eps$ is negligible as desired.

\end{proof}



\section{Consistency and applications}
\label{sec_CP_appl}

We conclude the paper by proving a consistency result and by discussing some explanatory examples. We begin by observing that, by arguing on real and imaginary part, we can easily remove the assumption, stated at the beginning of the paper, that the right-hand side $f$ and the initial data $g_0$ and $g_1$ are real valued, and we can therefore state Theorem \ref{theo_main} and Theorem \ref{theo_vw_well-posed} for complex-valued functions. This also means that when regularising we can allow complex valued mollifiers. 

\subsection{Consistency with the classical theory}

Our first task is to show that when the Cauchy problem is $C^\infty$ well-posed, i.e., a classical solution $u\in C^\infty([0,T]\times\R^n)$ exists given $g_0,g_1\in C^\infty_c(\R^n)$ and $f\in C^\infty([0,T]\times\R^n)$ with compact support with respect to $x$, then every very weak solution $u_\eps$ converges to the classical solution $u$. To see this we need to be more specific in the choice of the mollifiers. In particular we will distinguish between mollifiers (i.e. $\varphi\in C^\infty_c(\R^n)$ with $\int\varphi(x)\, dx=1$) and mollifiers with all the moments vanishing (i.e. $\psi\in \S(\R^n)$ with $\int\psi(x)\, dx=1$ and $\int x^\alpha \psi(x)\, dx=0$ for all $\alpha\in\N^n$). A mollifier with all the moments vanishing it is easily obtained as the inverse Fourier transform of a compactly supported function identically equal to $1$ around $0$. So, it does not have compact support and it is complex valued.
\begin{theorem}
\label{theo_consistency}
Let
\[
\begin{split}
\partial_t^2u-\sum_{i=1}^na_i(x)\partial^2_{x_i} u&=f(t,x),\quad t\in[0,T], x\in\R^n,\\
u(0,x)&=g_0,\\
\partial_tu(0,x)&=g_1,
\end{split}
\]
where $a\in B^\infty([0,T]\times\R^n)$, $a\ge 0$, $f\in C^\infty([0,T]\times\R^n)$ with compact support with respect to $x$ and $g_0,g_1\in C^\infty_c(\R^n)$. Hence,
\begin{itemize}
\item[(i)] the Cauchy problem above has a unique solution $u\in C^\infty([0,T]\times\R^n)$;
\item[(ii)] for every positive scale $\omega(\eps)$, for every mollifier $\varphi\ge 0$ and every mollifier $\psi$ with all the moments vanishing the Cauchy problem has a very weak solution, i.e., there exists a moderate net  $(u_\eps)_\eps$ such that 
\[
\begin{split}
\partial_t^2u_\eps-\sum_{i=1}^na_{i,\eps}(x)\partial^2_{x_i} u_\eps&=f_\eps(t,x),\quad t\in[0,T], x\in\R^n,\\
u_\eps(0,x)&=g_{0,\eps}\\
\partial_tu_\eps(0,x)&=g_{1,\eps},
\end{split}
\]
where  
\[
\begin{split}
a_{i,\eps}(x)&=(a_i\ast\varphi_{\omega(\eps)})(x),\\
f_\eps(t,x)&=(f(t,\cdot)\ast\psi_{\eps})(x),\\
g_{0,\eps}(x)&=(g_0\ast\psi_{\eps})(x),\\
g_{1,\eps}(x)&=(g_1\ast\psi_{\eps})(x).
\end{split}
\]
Furthermore, $(u_\eps)_\eps$ is bounded with respect to $\eps$ in all Sobolev norms, i.e., for all $k\in\N$ there exists a constant $c>0$ such that 
\[
\Vert u_\eps(t)\Vert_{H^{k+1}}\le c
\]
for all $t\in[0,T]$ and $\eps\in(0,1]$;
\item[(iii)] the net $u_\eps(t,\cdot)$ converges to $u(t,\cdot)$ in every Sobolev space uniformly with respect to $t\in[0,T]$.
\end{itemize}

\end{theorem}
\begin{proof}
\leavevmode
\begin{itemize}
\item[(i)] The $C^\infty$ well-posedness comes directly from Oleink's result.
\item[(ii)] Repeating the arguments of the previous section and, in particular, referring to Proposition \ref{prop_net_vww} we easily see that there exists a solution net $(u_\eps)_\eps$ such that
\[
\Vert u_\eps(t)\Vert_{H^{k+1}}^2\le C_{k,\eps}\biggl(\Vert g_{0,\eps}\Vert_{H^{k+2}}^2+\Vert g_{1,\eps}\Vert_{H^{k+1}}^2+\int_0^t \Vert f_\eps(s)\Vert_{H^{k+1}}^2\, ds\biggl),
\]
for all $k\in\N$. Since the coefficients $a_i$ belong to $B^\infty([0,T]\times\R^n)$ the estimate \eqref{Glaeser_eps} holds with a constant independent of $\eps$ and the $L^\infty$-norms of the derivatives of $a_{j,\eps}$ are as well bounded with respect to $\eps$. It follows that we can replace $C_{k,\eps}$ with a constant $C_k$, i.e., 
\[
\Vert u_\eps(t)\Vert_{H^{k+1}}^2\le C_{k}\biggl(\Vert g_{0,\eps}\Vert_{H^{k+2}}^2+\Vert g_{1,\eps}\Vert_{H^{k+1}}^2+\int_0^t \Vert f_\eps(s)\Vert_{H^{k+1}}^2\, ds\biggl),
\]
for all $k\in\N$ and $\eps\in(0,1]$. Note that from Proposition \ref{prop_reg_nets} the nets $g_{0,\eps}$ and $g_{1,\eps}$ are moderate and for all $k\in\N$ there exists a constant $c_k>0$ such that 
\[
\Vert g_{0,\eps}\Vert_{H^{k+2}}^2+\Vert g_{1,\eps}\Vert_{H^{k+1}}^2\le c_k
\]
for all $\eps\in(0,1]$. Analogously, the same holds for $\Vert f_\eps(s)\Vert_{H^{k+1}}^2$, uniformly in $t\in[0,T]$. By Sobolev's embedding theorem and iterated time-differentiation in the original equation, it follows that $u_\eps$ is moderate as a net of functions in $C^\infty([0,T]\times\R^n)$ and for all $k\in\N$ there exists a constant $c>0$ such that 
\[
\Vert u_\eps(t)\Vert_{H^{k+1}}\le c
\]
for all $t\in[0,T]$ and $\eps\in(0,1]$.
\item[(iii)] We now want to compare the solution net $(u_\eps)_\eps$ with the classical solution $u$. This means to compare the regularised Cauchy problem with solution $u_\eps$ with the classical one with solution $u$. We have that 
\[
\begin{split}
\partial_t^2(u_\eps-u)-\sum_{i=1}^na_{i,\eps}(x)\partial^2_{x_i} (u_\eps-u)&=\sum_{i=1}^n (a_{i,\eps}-a_i)(x)\partial_{x_i}^2u(t,x)+(f_\eps-f)(t,x),\\
(u_\eps-u)(0,x)&=g_{0,\eps}-g_0\\
(\partial_tu_\eps-\partial_t u)(0,x)&=g_{1,\eps}-g_1.
\end{split}
\]
By the estimate \eqref{est_u_H_k} and the arguments in (ii) we know that for all $k\in\N$ there exists $C_k>0$ such that
\begin{multline}
\label{est_H_k_cons}
\Vert (u_\eps-u)(t)\Vert_{H^{k+1}}^2\le C_{k}\biggl(\Vert g_{0,\eps}-g_0\Vert_{H^{k+2}}^2+\Vert g_{1,\eps}-g_1\Vert_{H^{k+1}}^2\\
+\int_0^t \Vert F_\eps(s)\Vert_{H^{k+1}}^2\, ds\biggl),
\end{multline}
where 
\[
F_\eps(t,x)=\sum_{i=1}^n (a_{i,\eps}-a_i)(x)\partial_{x_i}^2u(t,x)+(f_\eps-f)(t,x).
\]
Note that $g_0,g_1\in C^\infty_c(\R^n)\subseteq H^\infty(\R^n)$. Hence, by Proposition 2.7 in \cite{BO:92} we have that for all $k\in\N$ and for all $q\in\N$ there exists $c_{k,q}>0$ such that 
\beq
\label{est_neg_g}
\Vert g_{0,\eps}-g_0\Vert_{H^{k+2}}\le c_{k,q}\eps^q,\qquad \Vert g_{1,\eps}-g_1\Vert_{H^{k+1}}\le c_{k,q}\eps^q
\eeq
for all $\eps\in(0,1]$. Note that by the finite speed of propagation typical of hyperbolic equations $u$ is compactly supported with respect to $x$ and belongs to every Sobolev space. In addition, from Proposition \ref{prop_nets_asymp} we have that $(a_{i,\eps}-a_i)(x)\partial_{x_i}^2u(t,x)$ is negligible as a net of smooth functions in $x$ and more precisely, for all $k\in\N$ and for all $q\in\N$ there exists $c'_{k,q}>0$ such that 
\[
\Vert  (a_{i,\eps}-a_i)(\cdot)\partial_{x_i}^2u(t,\cdot)\Vert_{H^{k+1}}\le c'_{k,q}\eps^q
\]
for all $t\in[0,T]$ and $\eps\in(0,1]$. Analogously, the net $(f_\eps-f)_\eps$ is negligible as well and therefore for all $k\in\N$ and for all $q\in\N$ there exists $c'_{k,q}>0$ such that 
\beq
\label{est_neg_F}
\Vert( f_\eps-f)(t,\cdot)\Vert_{H^{k+1}}\le c'_{k,q}\eps^q
\eeq
for all $t\in[0,T]$ and $\eps\in(0,1]$. By inserting \eqref{est_neg_F} and \eqref{est_neg_g} into \eqref{est_H_k_cons} we conclude that for all $k\in\N$ and for all $q\in\N$ there exists a constant $c>0$ such that 
\[
\Vert (u_\eps-u)(t)\Vert_{H^{k+1}}\le c\eps^q,
\]
for all $t\in[0,T]$ and $\eps\in(0,1]$. It follows that $u_\eps(t,\cdot)$ converges to $u(t,\cdot)$ in every Sobolev space, uniformly with respect to time. 

\end{itemize}
\end{proof}

The rest of this section is devoted to some explanatory examples. For the sake of simplicity we work in one-space dimension.  
\subsection{Example 1}
Let us assume that $a$ is a $C^1$ function on $\R$ with discontinuous second order derivative and consider the Cauchy problem
\[
\begin{split}
\partial_t^2u-a(x)\partial^2_x u&=f(t,x),\quad t\in[0,T], x\in\R,\\
u(0,x)&=g_{0}\\
\partial_tu(0,x)&=g_{1}.
\end{split}
\]
For instance, let
\beq
\label{a_ex}
a(x)=\begin{cases}
\frac{1}{2}x^2\chi(x), & x>0,\\
0, & x\le 0,
\end{cases}
\eeq
where $\chi\in C^\infty_c(\R)$, $\chi\ge 0$, identically equal to 1 on a neighbourhood of $0$. It follows that $a\ge 0$ has compact support, belongs to $C^1$ and the second derivative is a jump-function. So, $a$ is a distribution in $\E'(\R)$ fulfilling the assumptions of Theorem \ref{theo_vw_well-posed}. Note that, even if we take right-hand side and initial data smooth we cannot in principle prove $C^\infty$ well-posedness. However, we can have very weak well-posedness with a suitable choice of mollifiers and regularising net. More precisely, if $\varphi$ is a positive mollifier then
\beq
\label{est_example}
\begin{split}
\Vert a\ast\varphi_{\omega(\eps)}\Vert_\infty&\le \Vert a\Vert_\infty \Vert \varphi_{\omega(\eps)}\Vert_1=\Vert a\Vert_\infty,\\
\Vert(a\ast\varphi_{\omega(\eps)})'\Vert_\infty&=\Vert a'\ast\varphi_{\omega(\eps)}\Vert_\infty\le \Vert a'\Vert_{\infty},\\
\Vert (a\ast\varphi_{\omega(\eps)})''\Vert_\infty&\le \Vert a''\Vert_\infty
\end{split}
\eeq
and 
\begin{multline}
\label{est_a_varphi}
|(a\ast \varphi_{\omega(\eps)})(x)-a(x)|=\biggl|\int_{\R}(a(x-\omega(\eps)y)-a(x))\varphi(y)\, dy\biggr|\\
\le c(a')\omega(\eps)\int_\R y\varphi(y)\, dy= c(a',\varphi)\omega(\eps),
\end{multline}
for all $x\in\R$ and $\eps\in(0,1]$. Let $(u_\eps)_\eps$ be the solution net of the Cauchy problem
\[
\begin{split}
\partial_t^2u_\eps-a_\eps(x)\partial^2_x u_\eps&=f_\eps(t,x),\quad t\in[0,T], x\in\R,\\
u_\eps(0,x)&=g_{0,\eps},\\
\partial_tu_\eps(0,x)&=g_{1,\eps},
\end{split}
\]
where $f\in C([0,T]),\E'(\R^n)$ and $g_0,g_1\in\E'(\R^n)$ and
\[
\begin{split}
a_{\eps}(x)&=(a\ast\varphi_{\omega(\eps)})(x),\\
f_\eps(t,x)&=(f(t,\cdot)\ast\varphi_{\eps})(x),\\
g_{0,\eps}(x)&=(g_0\ast\varphi_{\eps})(x),\\
g_{1,\eps}(x)&=(g_1\ast\varphi_{\eps})(x).\\
\end{split}
\]
In particular, from \eqref{est_u_H_k} it follows that
\[
\Vert u_\eps(t)\Vert_{H^{1}}^2\le C_{1,\eps}\biggl(\Vert g_{0,\eps}\Vert_{H^{2}}^2+\Vert g_{1,\eps}\Vert_{H^{1}}^2+\int_0^t \Vert f_\eps(s)\Vert_{H^{1}}^2\, ds\biggl),
\]
where the net $C_{1,\eps}$ depends on the estimates in \eqref{est_example} and therefore is bounded with respect to $\eps$. Hence, there exists $C_1>0$ such that 
\[
\Vert u_\eps(t)\Vert_{H^{1}}^2\le C_{1}\biggl(\Vert g_{0,\eps}\Vert_{H^{2}}^2+\Vert g_{1,\eps}\Vert_{H^{1}}^2+\int_0^t \Vert f_\eps(s)\Vert_{H^{1}}^2\, ds\biggl),
\]
for all $\eps\in(0,1]$. The same is true for the $H^2$-norm, i.e., there exists $C_2>0$ such that 
\[
\Vert u_\eps(t)\Vert_{H^{2}}^2\le C_{2}\biggl(\Vert g_{0,\eps}\Vert_{H^{3}}^2+\Vert g_{1,\eps}\Vert_{H^{2}}^2+\int_0^t \Vert f_\eps(s)\Vert_{H^{2}}^2\, ds\biggl),
\]
for all $\eps\in(0,1]$. Note that, from the estimate \eqref{est_U_n+1} and the fact that $U_{2}=\partial_t u$ we have that there exists a constant $C>0$ such that 
\beq
\label{est_u_ex}
\Vert u_\eps(t)\Vert_{L^2}^2\le C\biggl(\Vert g_{0,\eps}\Vert_{H^1}^2+\Vert g_{1,\eps}\Vert_{L^2}^2+\int_{0}^t\Vert f_\eps(s)\Vert_{L^2}^2\, ds\biggr)
\eeq
for all $t\in[0,T]$ and $\eps\in(0,1]$. Indeed, $C$ depends on $T$,  $\Vert a\ast\varphi_{\omega(\eps)}\Vert_\infty$ and $\Vert (a\ast\varphi_{\omega(\eps)})''\Vert_\infty$ and can therefore be chosen independently on $\eps$.

We now want to understand what happens to the very weak solution $(u_\eps)_\eps$ when we change the mollifier $\varphi$. Let us start by changing just the mollifier employed to regularise the coefficient $a$, i.e, let us compare
\[
\begin{split}
\partial_t^2u_\eps-a_\eps(x)\partial^2_x u_\eps&=f_\eps(t,x),\quad t\in[0,T], x\in\R,\\
u_\eps(0,x)&=g_{0,\eps},\\
\partial_tu_\eps(0,x)&=g_{1,\eps},
\end{split}
\]
with 
\[
\begin{split}
\partial_t^2\wt{u}_\eps-{\wt{a}}_\eps(x)\partial^2_x \wt{u}_\eps&=f_\eps(t,x),\quad t\in[0,T], x\in\R,\\
\wt{u}_\eps(0,x)&=g_{0,\eps},\\
\partial_t\wt{u}_\eps(0,x)&=g_{1,\eps},
\end{split}
\]
where
\[
\begin{split}
a_{\eps}(x)&=(a\ast\varphi_{\omega(\eps)})(x),\\
f_\eps(t,x)&=(f(t,\cdot)\ast\varphi_{\eps})(x),\\
g_{0,\eps}(x)&=(g_0\ast\varphi_{\eps})(x),\\
g_{1,\eps}(x)&=(g_1\ast\varphi_{\eps})(x).\\
\end{split}
\]
and 
\[
\wt{a}_{\eps}(x)=(a\ast\wt{\varphi}_{\omega(\eps)})(x).
\]
It follows that $v_\eps=u_\eps-\wt{u}_\eps$ solves the Cauchy problem
\[
\begin{split}
\partial_t^2 v_\eps-a_\eps(x)\partial^2_x v_\eps&=(a_\eps-\wt{a}_\eps)(x)\partial^2_x\wt{u}_\eps(t,x) \\
v_\eps(0,x)&=0,\\
\partial_tv_\eps(0,x)&=0.
\end{split}
\]
Arguing as above, from \eqref{est_u_ex} we have that 
\[
\Vert v_\eps(t)\Vert_{L^2}^2\le C_{1} \int_0^t \Vert (a_\eps-\wt{a}_\eps)(\cdot)\partial^2_x\wt{u}_\eps(s,\cdot)\Vert_{L^2}^2\, ds,
\]
for all $t\in[0,T]$ and $\eps\in(0,1]$.

We can now observe that, since $a$ as well as the mollifiers are compactly supported from \eqref{est_a_varphi} 
\[
|(a_\eps-\wt{a}_\eps)(x)|\le |a_\eps(x)-a(x)|+|\wt{a}_\eps(x)-a(x)|\le c\omega(\eps) 
\]
where $c=c(a',\varphi,\wt{\varphi})>0$. Hence
\[
\Vert a_\eps-\wt{a}_\eps\Vert_\infty=O(\omega(\eps)),
\]
for all $\eps\in(0,1]$. Let 
\beq
\label{est_lambda}
\Vert \partial^2_x\wt{u}_\eps(s,\cdot)\Vert_{L^2}=O(\lambda(\eps)),
\eeq
for $\eps\in(0,1]$. Note that by \eqref{est_u_ex} if 
\[
\biggl(\Vert g_{0,\eps}\Vert_{H^{3}}^2+\Vert g_{1,\eps}\Vert_{H^{2}}^2+\int_0^t \Vert f_\eps(s)\Vert_{H^{2}}^2\, ds\biggl)=O(\lambda(\eps))
\]
then \eqref{est_lambda} holds. It follows that,  
\[
\Vert u_\eps-\wt{u}_\eps\Vert_{L^2}^2=O(\omega^2(\eps)\lambda(\eps)), 
\]
where $\omega^2(\eps)\lambda(\eps)$ depends on the regularising scale for the coefficients $a$ and the regularisations of initial data and right-hand side. In particular,
\begin{itemize}
\item[(i)] If $\omega^2(\eps)\lambda(\eps)\to 0$ then $u_\eps-\wt{u}_\eps\to 0$ in $L^2$ as $\eps\to 0$;
\item[(ii)] If $\omega^2(\eps)\lambda(\eps)=O(1)$ then the net $u_\eps-\wt{u}_\eps$ has a subsequence which is weakly convergent to $0$.
\end{itemize}
In the next proposition we prove that, with a suitable choice of the regularising scale $\omega(\eps)$, we can guarantee that a change of mollifier will not have effect on the asymptotic properties of the solution $u_\eps$, in the sense that the difference of two solutions corresponding to different mollifiers will converge to $0$ in the $L^2$-norm. 
\begin{proposition}
\label{prop_ex_1}
Let 
\[
\begin{split}
\partial_t^2u-a(x)\partial^2_x u&=f(t,x),\quad t\in[0,T], x\in\R,\\
u(0,x)&=g_{0},\\
\partial_tu(0,x)&=g_{1},
\end{split}
\]
with $a$ as in \eqref{a_ex} and $g_0,g_1\in \E'(\R)$ and $f\in C([0,T],\E'(\R))$. Let $\varphi,\wt{\varphi}, \psi$ be mollifiers and let $u_\eps$ and $\wt{u_\eps}$ be the solutions nets of 
\[
\begin{split}
\partial_t^2u_\eps-a_\eps(x)\partial^2_x u&=f_\eps(t,x),\quad t\in[0,T], x\in\R,\\
u_\eps(0,x)&=g_{0,\eps},\\
\partial_tu_\eps(0,x)&=g_{1,\eps},
\end{split}
\]
and 
\[
\begin{split}
\partial_t^2{\wt u}_\eps-\wt{a}_\eps(x)\partial^2_x \wt{u}_\eps&=f_\eps(t,x),\quad t\in[0,T], x\in\R,\\
\wt{u}_\eps(0,x)&=g_{0,\eps},\\
\partial_t\wt{u}_\eps(0,x)&=g_{1,\eps},
\end{split}
\]
respectively, where 
\[
\begin{split}
a_{\eps}(x)&=(a\ast\varphi_{\omega(\eps)})(x),\\
f_\eps(t,x)&=(f(t,\cdot)\ast\psi_{\eps})(x),\\
g_{0,\eps}(x)&=(g_0\ast\psi_{\eps})(x),\\
g_{1,\eps}(x)&=(g_1\ast\psi_{\eps})(x).\\
\end{split}
\]
and 
\[
\wt{a}_{\eps}(x)=(a\ast\wt{\varphi}_{\omega(\eps)})(x).
\]
Hence, given $g_0$, $g_1$ and $f$ there exists a net $\omega(\eps)\to 0$ such that $\Vert u_\eps(t)-\wt{u}_\eps(t)\Vert_{L^2}\to 0$ uniformly with respect to $t\in[0,T]$.
\end{proposition}
\begin{proof}
By construction we have that 
\[
\Vert a_\eps-\wt{a}_\eps\Vert_\infty=O(\omega(\eps)),
\]
for all $\eps\in(0,1]$. We need to find $\lambda_\eps$ such that  
\[
\biggl(\Vert g_{0,\eps}\Vert_{H^{3}}^2+\Vert g_{1,\eps}\Vert_{H^{2}}^2+\int_0^t \Vert f_\eps(s)\Vert_{H^{2}}^2\, ds\biggl)=O(\lambda(\eps)).
\]
By the support theorem for compactly supported distributions we know that there exists $N$ large enough such that 
\[
\biggl(\Vert g_{0,\eps}\Vert_{H^{3}}^2+\Vert g_{1,\eps}\Vert_{H^{2}}^2+\int_0^t \Vert f_\eps(s)\Vert_{H^{2}}^2\, ds\biggl)=O(\eps^{-N}).
\]
This depends only on $g_0$, $g_1$ and $f$ and not on the mollifier $\psi$. Since
\[
\Vert u_\eps-\wt{u}_\eps\Vert_{L^2}^2=O(\omega^2(\eps)\lambda(\eps)), 
\]
 by choosing $\omega^2(\eps)=\eps^{N+1}$ we have that $u_\eps-\wt{u}_\eps$ tends to $0$ in $L^2(\R)$ uniformly with respect to $t$.  
\end{proof}

\subsection{Example 2}
Let us now study the Cauchy problem 
\[
\begin{split}
\partial_t^2u-H(x)\partial^2_x u&=f(t,x),\quad t\in[0,T], x\in\R,\\
u(0,x)&=g_{0},\\
\partial_tu(0,x)&=g_{1},
\end{split}
\]
where $H$ is the Heaviside function ($H(x)=0$ for $x<0$ and $H(x)=1$ for $x\ge 0$). Clearly $H(x)$ is not a $C^2$ function and its first and second derivatives are highly singular. As in the previous example we want to see what happens in our solution net $(u_\eps)_\eps$ when we change the mollifier employed in regularising the coefficient $H(x)$. In detail, let $g_0,g_1\in \E'(\R)$ and $f\in C([0,T],\E'(\R))$. Let $\varphi,\wt{\varphi}, \psi$ be mollifiers and let $u_\eps$ and $\wt{u_\eps}$ be the solutions nets of 
\[
\begin{split}
\partial_t^2u_\eps-H_\eps(x)\partial^2_x u&=f_\eps(t,x),\quad t\in[0,T], x\in\R,\\
u_\eps(0,x)&=g_{0,\eps},\\
\partial_tu_\eps(0,x)&=g_{1,\eps},
\end{split}
\]
and 
\[
\begin{split}
\partial_t^2{\wt u}_\eps-\wt{H}_\eps(x)\partial^2_x \wt{u}_\eps&=f_\eps(t,x),\quad t\in[0,T], x\in\R,\\
\wt{u}_\eps(0,x)&=g_{0,\eps},\\
\partial_t\wt{u}_\eps(0,x)&=g_{1,\eps},
\end{split}
\]
respectively, where 
\[
\begin{split}
H_{\eps}(x)&=(H\ast\varphi_{\omega(\eps)})(x),\\
f_\eps(t,x)&=(f(t,\cdot)\ast\psi_{\eps})(x),\\
g_{0,\eps}(x)&=(g_0\ast\psi_{\eps})(x),\\
g_{1,\eps}(x)&=(g_1\ast\psi_{\eps})(x).\\
\end{split}
\]
and 
\[
\wt{H}_{\eps}(x)=(H\ast\wt{\varphi}_{\omega(\eps)})(x).
\]
Because of the high singularity of the derivatives of the coefficient $H$ we cannot expect that the net $u_\eps-\wt{u}_\eps$ will converge to $0$ in some sense for arbitrary choice of mollifiers $\varphi$ and $\wt{\varphi}$  but we can at least provide a precise analysis of how the solution net depends on the mollifiers and the net $\omega(\eps)$. As observed at the beginning of the paper 
\[
\Vert u_\eps(t)\Vert_{L^2}^2\le C_{2,\eps}\biggl(\Vert g_{0,\eps}\Vert_{H^1}^2+\Vert g_{1,\eps}\Vert_{L^2}^2+\int_{0}^t\Vert f_\eps(s)\Vert_{L^2}^2\, ds\biggr),
\]
where 
\[
C_{2,\eps}={\rm e}^{\max(2\Vert H_\eps''\Vert_\infty,2)T}\max(\Vert H_\eps\Vert_{\infty},1).
\]
Now
\[
\Vert H_\eps''\Vert=\Vert \delta\ast (\varphi_{\omega(\eps)})'\Vert_\infty=\omega(\eps)^{-2}\Vert \varphi'\Vert_{\infty}
\]
and
\[
\Vert H_\eps\Vert_{\infty}\le 1
\]
for $\varphi\ge 0$ with $\int_\R\varphi\, dx=1$. It follows that 
\[
C_{2,\eps}={\rm e}^{2T\omega(\eps)^{-2}\Vert \varphi'\Vert_{\infty}}
\]
and to get moderate estimates we need to set
\[
\omega(\eps)^{-2}=\ln(\eps^{-1}).
\]
for $\eps$ small enough. Hence,
\[
\Vert u_\eps(t)\Vert_{L^2}^2\le \eps^{-2T\Vert \varphi'\Vert_{\infty}}\biggl(\Vert g_{0,\eps}\Vert_{H^1}^2+\Vert g_{1,\eps}\Vert_{L^2}^2+\int_{0}^t\Vert f_\eps(s)\Vert_{L^2}^2\, ds\biggr).
\]
This proves that the solution net $(u_\eps)_\eps$ depends on the $L^\infty$-norm of the first derivative of the mollifier $\varphi$. Therefore , if we choose $\varphi$ and $\wt{\varphi}$ with $\Vert \varphi'\Vert_{\infty}=\Vert \wt{\varphi}'\Vert_{\infty}$ and $\omega(\eps)$ as above the corresponding solution nets will have the same behaviour in $\epsilon$.


\begin{thebibliography}{CDGS79}

\bibitem[ART19]{ART:19}
 {A.  Altybay, M. Ruzhansky and N. Tokmagambetov,}
 Wave equation with distributional propagation speed and mass term: numerical simulations. 
 \newblock {\em Appl. Math. E-Notes}, 19, 552--562, 2019.
 

\bibitem[AdHSU08]{AdHSU:08}
{F. Andersoon, M. de Hoop, H. Smith and G. Uhlmann.}
\newblock A multi-scale approach to hyperbolic evolution equations with limited smoothness.
\newblock {\em Comm. Partial Differential Equations}, 33(4-6):988--1017, 2018.

\bibitem[BO92]{BO:92}
H.~Biagioni and M.~Oberguggenberger.
\newblock Generalized solutions to the {K}orteweg - de {V}ries and the
  regularized long-wave equations.
\newblock {\em SIAM J. Math Anal.}, 23(4):923--940, 1992.

 \bibitem[B]{B}
{\sc M.~D.~Bronshtein.}
The Cauchy problem for hyperbolic operators with characteristics of variable multiplicity.
(Russian)
\newblock {\em Trudy Moskov. Mat. Obshch.}, 41:83--99, 1980; Trans. Moscow Math.
Soc. 1:87--103, 1982.

\bibitem[BdHSU11]{BdHSU:11}
V. Brytik, M. de Hoop, H. Smith and G. Uhlmann.
\newblock Decoupling of modes for the elastic wave equation in media of limited smoothness.
\newblock {\em Comm. Partial Differential Equations}, 36(10):1683--1693, 2011.

\bibitem[CdHKU19]{CdHKU:19}
{P. Caday, M. de Hoop, V. Katsnelson and G. Uhlmann.}
\newblock Scattering control for the wave equation with unknown wave speed.
\newblock {\em Arch. Ration. Mech. Anal.}, 231(1):409--464, 2019.

\bibitem[CK]{ColKi:02}
{F.~Colombini and T.~Kinoshita.}
\newblock On the Gevrey well posedness of the Cauchy problem for weakly hyperbolic equations of higher order.
\newblock {\em J. Diff.  Eq.}, 186:394--419, 2002.

\bibitem[CK02]{ColKi:02-2}
{F.~Colombini and T.~Kinoshita.}
\newblock
On the Gevrey wellposedness of the Cauchy problem for weakly hyperbolic equations of 4th order.
\newblock {\em Hokkaido Math. J.}, 31:39--60, 2002.

\bibitem[CS]{CS}
{F. Colombini and S. Spagnolo}.
An example of a weakly hyperbolic Cauchy problem not well posed in $C^\infty$.
{\em Acta Math.}, 148:243--253, 1982.

\bibitem[CDS]{CDS}
{F. Colombini, E. De Giorgi and S. Spagnolo}.
Sur les \'equations hyperboliques avec des coefficients qui ne d\'ependent que du temps.
{\em Ann. Scuola Norm. Sup. Pisa Cl. Sci.}, 6:511--559, 1979.

\bibitem[CJS2]{CJS2}
{F. Colombini, E. Jannelli and S. Spagnolo}.
Nonuniqueness in hyperbolic Cauchy problems.
{\em Ann. of Math.}, {126}:495--524, 1987.

\bibitem[dAKi05]{dAKi:05}
{P. d'Ancona and T. Kinoshita}.
On the wellposedness of the Cauchy problem for weakly hyperbolic equations of higher
order.
{\em Math. Nachr.}, 278:1147--1162, 2005.

\bibitem[dAS98]{dASpa:98}
{P. d'Ancona and S. Spagnolo}.
Quasi-Symmetrization of Hyperbolic Systems
and Propagation of the Analytic Regularity.
{\em Bollettino U.M.I.}, 8(1):169--185, 1998.

\bibitem[dHHO08]{dHHO:08}
M. de Hoop, G. H\"ormann and M. Oberguggenberger.
\newblock Evolution systems for paraxial wave equations of Schr\"odinger-type with non-smooth coefficients.
\newblock {\em J. Differential Equations}, 245(6):1413--1432, 2008.

\bibitem[dHHSU12]{dHHSU:12}
M. de Hoop, S. Holman, H. Smith and G. Uhlmann.
\newblock Regularity and multi-scale discretization of the solution construction of hyperbolic evolution equations with limited smoothness.
\newblock {\em Appl. Comput. Harmon. Anal.}, 33(3):330--353, 2012. 

\bibitem[D]{Dencker} 
{N. Dencker}, 
On the propagation of singularities for pseudo-differential operators with characteristics of 
variable multiplicity. 
Comm. Partial Differential Equations, {17}:1709--1736, 1992.
 
 
\bibitem[GJ17]{GarJ}
{C. Garetto and C. J\"ah}
\newblock Well-posedness of hyperbolic systems with multiplicities and smooth coefficients
\newblock{\em Math. Ann}, 369(1-2):441--485, 2017.
 
\bibitem[GR12]{GR:11}
C.~Garetto and M.~Ruzhansky.
\newblock On the well-posedness of weakly hyperbolic equations with
  time-dependent coefficients.
\newblock {\em J. Differential Equations}, 253(5):1317--1340, 2012.

\bibitem[GR13]{GR:12}
C.~Garetto and M.~Ruzhansky.
\newblock Weakly hyperbolic equations with non-analytic coefficients and lower
  order terms.
\newblock {\em Math. Ann.}, 357(2):401--440, 2013.

\bibitem[GR14]{GR:14}
C.~Garetto and M.~Ruzhansky.
\newblock Hyperbolic  second order equations with non-regular time dependent coefficients
\newblock {\em Arch. Ration. Mech. Anal.}, 217(1):113--154, 2015. 

\bibitem[GR14b]{GarRuz:3}
{C.~Garetto and M.~Ruzhansky.}
\newblock A note on weakly hyperbolic equations with analytic principal part
\newblock {\em J. Math. Anal. Appl.},  412(1):1--14, 2014.

 \bibitem[GR17]{GarRuz:7}
{C.~Garetto and M.~Ruzhansky.}
\newblock $C^\infty$ well-posedness of hyperbolic systems with multiplicities
\newblock {\em Ann. Mat. Pura  Appl}, 196(5):1819--1834, 2017.

\bibitem[GJR18]{GJR1}
{C. Garetto, C. J\"ah and M. Ruzhansky.}
\newblock Hyperbolic systems with non-diagonalisable principal part and
  variable multiplicities, {I}: well-posedness. 
\newblock {\em Math. Annalen}, 372(3-4):1597-1629, 2018. 


\bibitem[GJR20]{GJR2}
{C. Garetto, C. J\"ah and M. Ruzhansky.}
Hyperbolic systems with non-diagonalisable principal part and variable multiplicities, {II}: microlocal analysis. 
{\em arxiv.org/abs/2001.04709}, 2020. 

\bibitem[GKOS01]{GKOS:01}
M.~Grosser, M.~Kunzinger, M.~Oberguggenberger, and R.~Steinbauer.
\newblock {\em Geometric Theory of generalized Functions with Applications to General Relativity}, volume 537 of
{Mathematics and its Applications}.
\newblock Kluwer Acad. Publ., Dordrecht, 2001.

\bibitem[Hor93]{Hord} 
{L. H\"ormander}.
Hyperbolic systems with double characteristics, 
{\em Comm. Pure Appl. Math.}, 46:261--301, 1993.

\bibitem[Hor97]{Hor97}
L. H\"ormander.
\newblock Lectures on nonlinear hyperbolic differential equations.  
\newblock Mathematics and Applications, 26, {\em Springer-Verlag}, Berlin, 1997.

\bibitem[HdH01]{HdH01}
G.~H{\"o}rmann and M.~V. de~Hoop.
\newblock Microlocal analysis and global solutions of some hyperbolic equations
  with discontinuous coefficients.
\newblock {\em Acta Appl. Math.}, 67(2):173--224, 2001.

\bibitem[HdH02]{HdH02}
G.~H{\"o}rmann and M.~V. de~Hoop.
\newblock Detection of wave front set perturbations via correlation: foundation
  for wave-equation tomography.
\newblock {\em Appl. Anal.}, 81(6):1443--1465, 2002.

\bibitem[IP]{IvPet}
{V. Ya. Ivrii and V. M. Petkov}.
Necessary conditions for the correctness of the Cauchy problem
for non-strictly hyperbolic equations.
(Russian) {\em Russian Math. Surveys}, {29}:3--70, 1974.

\bibitem[J09]{J:09}
E. Jannelli.
\newblock The hyperbolic symmetrizer: theory and applications.
\newblock{\em in Advances in Phase Space Analysis of PDEs}, Birkh\"auser, 78:113--139, 2009.

\bibitem[JT]{JT}
E. Jannelli and G. Taglialatela.
\newblock Homogeneous weakly hyperbolic equations with time dependent analytic coefficients.
\newblock{\em J. Differential Equations}, 251:995--1029, 2011.

\bibitem[KY]{KY:06}
{K. Kajitani and Y. Yuzawa.}
\newblock  The Cauchy problem for hyperbolic systems with H\"older continuous coefficients with respect to the time variable.
{\em Ann. Sc. Norm. Super. Pisa Cl. Sci.}, {5(4)}:465--482, 2006.

\bibitem[KR]{KR2}
{I. Kamotski and M. Ruzhansky}.
Regularity properties, representation of solutions and 
spectral asymptotics of systems with multiplicities, 
{\em Comm. Partial Differential Equations},
{32}:1--35, 2007.


\bibitem[KS]{KS}
T. Kinoshita and S. Spagnolo,
\newblock Hyperbolic equations with non-analytic coefficients.
\newblock{\em Math. Ann.} 336:551--569, 2006.

\bibitem[KO]{9}
{V. V. Kucherenko and Yu. V. Osipov}.
The Cauchy problem for nonstrictly hyperbolic equations. (Russian)
{\em Mat. Sb. (N.S.)}, {120}(162):84--111, 1983.

\bibitem[LO91]{LO:91}
F.~Lafon and M.~Oberguggenberger.
\newblock Generalized solutions to symmetric hyperbolic systems with
  discontinuous coefficients: the multidimensional case.
\newblock {\em J. Math. Anal. Appl.}, 160(1):93--106, 1991.

\bibitem[MR]{MR} 
{M. Mascarello and L. Rodino}.
   Partial differential equations with multiple characteristics.
   Akademie Ver., 1997.
 

\bibitem[MU1]{MU:1}
{R. B. Melrose and G. A. Uhlmann}.
{Lagrangian intersection and the
Cauchy problem}, {\em Communs. Pure and Appl. Math.}, {32}:483--519, 1979.

\bibitem[MU2]{MU:2}
{R. B. Melrose and G. A. Uhlmann}.
{Microlocal structure of involutive
conical refraction}, {\em Duke Math. J.}, {46}:571--582, 1979.

\bibitem[MRT19]{MRT:19}
{J. C. Munoz, M. Ruzhansky and N. Tokmagambetov}.
{Wave propagation with irregular dissipation and applications to acoustic problems and shallow waters}, {\em J. Math. Pures Appl.}, 9(123), 127--147, 2019.

\bibitem[Obe92]{Oberguggenberger:Bk-1992}
M.~Oberguggenberger.
\newblock {\em Multiplication of distributions and applications to partial
  differential equations}, volume 259 of {\em Pitman Research Notes in
  Mathematics Series}.
\newblock Longman Scientific \& Technical, Harlow, 1992.

\bibitem[O70]{O70}
O. A. Oleinik.
\newblock On the Cauchy problem for weakly hyperbolic equations.
\newblock{\em  Comm. Pure Appl. Math.} 23:569--586, 1970.

\bibitem[OT84]{OT:84}
{Y. Ohya and S. Tarama.}
\newblock The Cauchy Problem with multiple characteristics in the Gevery class - H\"older coefficients in $t$.
\newblock{\em Hyperbolic Equations and related topics. Kataka/Kioto}, 273--306, 1984.

\bibitem[PP]{PP}
{C. Parenti and A. Parmeggiani}.
On the Cauchy problem for hyperbolic operators with double characteristics. 
{\em Comm. Partial Diff. Eq.}, {34}:837--888, 2009.

\bibitem[RT]{RT}
{Ruzhansky M. and Tokmagambetov N.}
Wave equation for operators with discrete spectrum and irregular propagation speed.
{\em Arch. Ration. Mech. Anal.}, 226: 1161--1207, 2017.

\bibitem[RY]{RY}
{Ruzhansky M. and Yessirkegenov N.}
Very weak solutions to hypoelliptic wave equations. 
{\em J. Differential Equations}, 268 2063--2088, 2020.

\bibitem[SW]{SW}
{M. E. Sebin and J. Wirth}.
On a wave equation with singular dissipation.
{\em https://arxiv.org/abs/2002.00825}, 2020.


\bibitem[ST06]{ST06}
S. Spagnolo and G. Taglialatela,
\newblock Some inequalities of Glaeser-Bronstein type.  
\newblock{\em Rend. Lincei Mat. Appl.},  17:367--375, 2006.

\bibitem[ST07]{ST}
S. Spagnolo and G. Taglialatela,
\newblock Homogeneous hyperbolic equations with coefficients depending on one space variable.   
\newblock{\em J. Differential Equations}, 4(3):533--553, 2007.

\bibitem[Y]{Yagd} 
{K. Yagdjian}.
The Cauchy problem for hyperbolic operators.
Multiple characteristics. Micro-local approach. 
Mathematical Topics, 12. Akademie Verlag, Berlin, 1997.

\end{thebibliography}
\end{document}